\documentclass[a4paper]{article}

\usepackage{amsfonts}
\usepackage{amssymb,amsthm}
\usepackage{graphicx}
\usepackage{amsmath}
\usepackage[a4paper]{geometry}
\usepackage{comment}
\usepackage{enumerate}
\usepackage{bigints}
\usepackage{dsfont}
\usepackage{hyperref}
\usepackage{mathtools}
\usepackage{enumitem}

\newtheorem{thm}{Theorem}

\newtheorem{lem}[thm]{Lemma}
\newtheorem{prop}[thm]{Proposition}
\newenvironment{assump}[1][Assumption]{\noindent\textbf{#1}}{\ \rule{0em}{0em} }

\numberwithin{thm}{section}
\numberwithin{equation}{section}

\DeclareMathOperator{\1}{{\mathds 1}}

\DeclareMathOperator{\R}{{\mathbb R}}

\DeclareMathOperator{\Z}{{\mathbb Z}}
\DeclareMathOperator{\N}{{\mathbb N}}
\DeclareMathOperator{\supp}{supp}

\providecommand{\abs}[1]{\lvert #1 \rvert}
\providecommand{\norm}[1]{\lVert #1 \rVert}
\providecommand{\bnorm}[1]{{\Bigl\lVert #1 \Bigr\rVert}}

\renewcommand{\le}{\leqslant}

\DeclareMathOperator{\Log}{Log}
\providecommand{\qsub}{  \!\!\!\!\!\!\!\! }
\providecommand{\qqsub}{ \qsub \qsub \qsub \qsub }

\providecommand{\FT}{{\cal F}}
\providecommand{\FTT}[1]{{\cal F} \left[ #1\right]}
\providecommand{\FI}{{\cal F}^{-1}}
\providecommand{\FII}[1]{{\cal F}^{-1} \left[ #1 \right]}

\DeclareMathOperator{\Deltaup}{{\mathrm{\Delta}}}

\begin{document}

\title{Efficient nonparametric inference for discretely observed compound Poisson processes}

\author{Alberto J. Coca \footnote{The author is grateful to Richard Nickl for the numerous discussions leading to the completion of this work, to the anonymous referees for their insightful remarks and to \textit{Fundaci\'on ``La Caixa''}, \textit{EPSRC} (grant EP/H023348/1 for the Cambridge Centre for Analysis) and \textit{Fundaci\'on Mutua Madrile\~na} for their generous support. } \\
\\ 
\textit{University of Cambridge \footnote{Statslab, Department of Pure Mathematics and Mathematical Statistics, University of Cambridge, CB30WB, Cambridge, UK. Email: a.coca-cabrero@maths.cam.ac.uk}} \\
\date{\today}}

\maketitle

\begin{abstract}

A compound Poisson process whose parameters are all unknown is observed at finitely many equispaced times. Nonparametric estimators of the jump and L\'evy distributions are proposed and functional central limit theorems using the uniform norm are proved for both under mild conditions. The limiting Gaussian processes are identified and efficiency of the estimators is established. Kernel estimators for the mass function, the intensity and the drift are also proposed, their asymptotic properties including efficiency are analysed, and joint asymptotic normality is shown.  Inference tools such as confidence regions and tests are briefly discussed.



\medskip



\end{abstract}

\section{Introduction}\label{SecIntroduction}

Let $(N_t)_{t\geq 0}$ be a one-dimensional Poisson process with intensity $\lambda>0$ and let $Y_1, Y_2, \ldots$ be a sequence of independent and identically distributed real-valued random variables with common distribution $F$. Assume this sequence is independent of the Poisson process and let $\gamma \in \R$. Then
a one-dimensional compound Poisson process with jump distribution $F\!$, intensity $\lambda$ and drift $\gamma$ can be written as
\vspace{-0.05cm}
\begin{equation}\label{CPPFirst}
	X_t= \gamma t + \sum_{j=1}^{N_t} Y_j, \qquad t\geq 0,
	\vspace{-0.05cm}
\end{equation}
where an empty sum is zero by convention so, in particular, the process always starts at zero. 

Compound Poisson processes are used in numerous applications and in many of them some of the defining parameters, if not all, are unknown. Therefore statistical inference on them is of great interest. In some applications such as insurance or queuing theory the process $(X_t)_{t\geq 0}$ may be observed continuously up to some time $T$. Hence the jumps $Y_1, Y_2, \ldots, Y_{N_T}$, representing the size of customer claims or the size of groups arriving at a queue, are directly observed. Under this setting standard tools from statistics can be used to make inference on $F$ and $\lambda$. Conditional on $m=N_T$, the process $\sqrt{m} \, (F_m-F)$, where $F_m$ is the empirical distribution function of the jumps, is approximately an $F$-Brownian bridge in view of Donsker's theorem (cf. \cite{D99} and \cite{AvdVW96}). This allows construction of inference tools of practical importance such as global confidence bands around $F$, goodness-of-fit tests and two-sample tests. In other situations, $(X_t)_{t\geq 0}$ is not observed continuously but discretely. For example, in chemistry or storage theory only measurements every ${\Deltaup}>0$ amount of time may be available up to time $T$. Therefore we do not directly observe the jumps representing the impact of chemical reactions or the amount of a natural resource being lost or gained in a storage system. Instead we observe $X_{{\Deltaup}}, X_{2{\Deltaup}}, \ldots, X_{n{\Deltaup}}$, where $\Deltaup$ is fixed and $n=\lfloor T/\Deltaup \rfloor$. This is called the low-frequency observation scheme and when making inference on the parameters we are confronted with a statistical inverse problem.

In Section 3.2 of \cite{NR12}, Nickl and Rei{\ss} conjecture that in the second setting it should be possible to develop a result similar to Donsker's theorem, and this is the main contribution of our results. We construct kernel-based estimators $\widehat F_n$ and $\widehat{N}_n$ of $F$ and $N:=\lambda F$ and prove that as $n\to \infty$
\begin{equation*}
	\sqrt{n} \,  \Big(\widehat{F}_n - F \Big)  \to^{\mathcal{L}} \mathbb G^{F} \, \, \mbox{ in } \ell^{\infty}(\R) \quad \mbox{ and } \quad \sqrt{n} \, \Big(\widehat{N}_n - N \Big)  \to^{\mathcal{L}} \mathbb B^{N} \, \, \mbox{ in } \ell^{\infty}(\R),
\end{equation*}
where $\mathbb G^{F}$ and $\mathbb B^{N}$ are centred Gaussian Borel random variables in the Banach space of uniformly bounded functions on $\R$, and $\to^{\mathcal{L}}$ denotes convergence in law in that space. We identify their covariance structures and these coincide with the information-theoretic lower bounds developed by Trabs in \cite{T15}, so our estimators are asymptotically efficient. The inverse nature of the problem translates into convoluted expressions for the covariances. However, the limiting processes can be shown to be of Brownian bridge and of Brownian motion type, respectively. When $\lambda {\Deltaup}$ is small, the jumps $Y_1, Y_2, \ldots$ are more likely to be observed and, in particular, we show that the covariances can be approximated by
\begin{equation*}
	\frac{1}{\lambda {\Deltaup}} \big( F(\min\{s,t\}) - F(s) F(t)\big)  \quad \mbox{ and } \quad \frac{1}{{\Deltaup}} \, N(\min\{s,t\}), \quad s,t\in \R.
\end{equation*}
The first expression agrees with Donsker's theorem because $\lambda {\Deltaup} n$ is the expected number of jumps of the compound Poisson process up to time $T$. Intuitively, the second limiting process is of Brownian motion type due to the stochasticity $\widehat{N}_n$ carries through the lack of knowledge of $\lambda$. 


In our proofs $F$ may have an absolutely continuous component with a mild uniform logarithmic modulus of continuity. We also allow it to have a discrete component and assume it has no atom at the origin, which does not pose any modelling sacrifice and ensures identifiability of the model. In turn, it makes $X$ a L{\'e}vy process and guarantees there is a notion of efficiency as shown in \cite{T15}. We will therefore work with the increments $Z_k:=X_{k{\Deltaup}}-X_{(k-1){\Deltaup}}$, $k=1, \ldots, n$, and exploit that they are independent and identically distributed with the same distribution as $X_{\Deltaup}$. On the discrete component we further assume it is supported in $\varepsilon \times \mathbb{Z}\setminus\{0\}$ for some $\varepsilon>0$ fixed. Therefore, exact cancellations between jumps in the expression for $X_{\Deltaup}$ may occur: e.g. if $\gamma=0$, zero increments do not necessarily correspond to no jumps in the corresponding $\Deltaup$-time interval. This type of setting has not yet been considered in the related literature. The kernel estimators $\widehat F_n$ and $\widehat{N}_n$ include a novel and general way to estimate the mass function of the discrete component from which simple tests for the presence of each component immediately follow. 
We also propose new estimators for the intensity and the drift, analyse their asymptotic properties including efficiency and show joint asymptotic normality of all the estimators. 

We construct the estimators via the so-called spectral approach introduced by Belomestny and Rei{\ss} in \cite{BR06}. The characteristic function of the increments, $\varphi$, is related to the jump distribution $F$ through a nonlinear relationship. This can be inverted with the aid of a suitable band-limited kernel and the estimators follow by plugging-in the empirical counterpart of $\varphi$. As hinted at by \eqref{CPPFirst}, after linearisation the problem is of deconvolution type, where the law of the error is that of the increments themselves. Nickl and Rei{\ss} \cite{NR12} call it an auto-deconvolution problem following the work of Neumann and Rei{\ss} in \cite{NR09}, where the ill-posedness of the problem for general L\'evy processes was first studied. 
Then, in the context of compound Poisson processes the non-linear inverse problem is not ill-posed in terms of convergence rates because $\inf_{u\in \R} |\varphi(u) |>0$. This justifies why consistent estimators at rate $1/\sqrt{n}$ exist and it implies the bandwidth $h_n$ of the kernel can decay exponentially with $n$. The regularity of the problem is also reflected in the rate of convergence of our estimator of the drift $\gamma$, which equals $h_n/\sqrt{n}$. This rate may at first appear surprising but stems from the fact that for continuous observations we can perfectly learn the drift (it is the slope of the observed process in the region with no jumps), and that the low-frequency regime is only marginally ill-posed in terms of information loss.


The first to consider the problem above were Buchmann and Gr{\"u}bel in \cite{BG03}, who coined it \textit{decompounding}. They assumed the unknown process at hand has strictly positive jumps and no drift, and treated the case of discrete and general jumps separately, assuming $\lambda$ is unknown and known, respectively. They constructed estimators of the probability mass function and of $F$ starting from the convolution-type relationship relating the distributions of $X_{\Deltaup}$ and $Y$ arising from \eqref{CPPFirst} and showed that their estimators of $F$ satisfy functional central limit theorems under an exponentially weighted sup-norm. This is the usual sup-norm only when $\lambda \Deltaup<\log(2)\approx 0.69$ and, otherwise, the exponential weight translates into estimation of $F$ only in a bounded set around the origin, thus not addressing what occurs in the tail. As shown by Coca \cite{C16}, the limiting process in the discrete case is $\mathbb{G}^{F}$ whilst in the general case it is not and their general estimator is not efficient. 
They overcame some practical limitations of their discrete estimator in \cite{BG04}.

After the work in \cite{NR09} and based on the techniques introduced by Gin\'e and Zinn \cite{GZ84} and further developed by Gin\'e and Nickl \cite{GN08}, Nickl and Rei{\ss} \cite{NR12} proved a functional central limit theorem for a class of L{\'e}vy processes with Blumenthal--Getoor index zero. The associated low polynomial decay of the characteristic function of the increments means that this is effectively the largest class for which the ill-posedness of the inverse problem allows this type of theorem to hold in the low-frequency observation scheme. Due to the potential non-integrable singularity at the origin in the L{\'e}vy measure, $N$ is only well-defined in $\R^-$, so for $\R^+$ the integral of the right tail should be considered instead. Therefore, they estimate this generalised L{\'e}vy distribution 
in the norm $L^{\infty}((-\infty, -\zeta] \cup [\zeta,\infty))$, for some $\zeta>0$ fixed. Their limiting Gaussian process is optimal as shown in \cite{T15}, and their estimator is efficient. Their result applies to compound Poisson processes but does not address estimation around the origin. 
Nevertheless, their ideas are the starting point to prove our results, 
which imply that in the compound Poisson case an estimator based on $\varphi$ and not on any of its derivatives can estimate $F$ and $N$ efficiently. With this estimator we can also relax their finite second moment assumption to a mild logarithmic one. 

Other functional central limit theorems in related settings have been developed. In the high-frequency regime, ${\Deltaup}={\Deltaup}_n\to 0$ and $n {\Deltaup}_n \to \infty$, the inverse nature of the problem vanishes as the number of observations grows. Consequently, Nickl, Rei\ss, S\"ohl and Trabs \cite{NRST16} prove a Donsker type of theorem on $\R$ for functionals of the L{\'e}vy measure for a much larger class of L\'evy processes that may carry a diffusion component. They also derive goodness of fit tests complementing the work of Rei{\ss} \cite{R13} on testing L\'evy processes. Previously Buchmann \cite{B09} showed a functional central limit theorem with an exponentially weighted sup-norm using continuous, although perhaps incomplete, observations. When the assumption of exponentially distributed interarrival times in \eqref{CPPFirst} is dropped, B{\o}gsted and Pitts \cite{BP10} show that the results in \cite{BG03} still hold. 

Nonparametric estimation of the jump density of compound Poisson processes and superclasses of them has also received a great deal of attention recently. Estimation of the former was initiated by van Es, Gugushvili and Spreij \cite{VEGS07}, who construct an estimator by inverting the characteristic function of the increments when low-frequency observations are available. Using the same idea Gugushvili \cite{G09,G12} subsequently extends the results in several directions, including adding a diffusion component to the observed process. Chen, Delaigle and Hall \cite{CDH10} generalise the noise component to a symmetric stable process, also estimate the distribution function under the sup-norm and obtain minimax rates of the problem complementing those in \cite{NR09}. Using prices of financial options and assuming the logarithm of the prices of financial assets is a compound Poisson process with a diffusion component, \cite{BR06,CT04,S14,ST14} perform density estimation and develop confidence sets. Also using the spectral approach, Kappus \cite{K14} adaptively estimates the jump density of a L{\'e}vy process. Starting from the convolution relationship arising from \eqref{CPPFirst}, Duval \cite{D13a} performs adaptive density estimation for compound Poisson processes in the high-frequency setting, Comte, Duval and Genon-Catalot \cite{CDGC14} deal with low-frequency observations and Duval \cite{D14} studies the identifiability and estimation of these processes when ${\Deltaup}={\Deltaup}_n \to \infty$. Duval \cite{D13b} extends the first result to renewal reward process and Comte, Duval, Genon-Catalot and Kappus \cite{CDGCK15} extend it to mixed compound Poisson process. We refer the reader to \cite{BCGCMR15} for a recent and exhaustive account of inference on L\'evy processes. More recently, Gugushvili, van der Meulen and Spreij \cite{GMS15a,GMS15b} make Bayesian inference on the density for multidimensional compound Poisson processes from low- and high-frequency observations, respectively.

This paper is divided into two main blocks: Sections \ref{SecMainResults} and \ref{SecProofs}. The first is split into three subsections introducing the notation, the precise assumptions and the estimators, and the main results. The main results are divided into three subsections about marginal estimation, efficiency remarks and joint estimation. Section \ref{SecProofs} is devoted to the proofs starting with auxiliary results, of which some may be of independent interest, and ending with the proofs of the main results.

%
%
%
%
%
%

\section{Main results} \label{SecMainResults}

\subsection{Definitions and notation} \label{SecDefNot}

Denote by $(\Omega, \mathcal{A}, \Pr)$ the probability space on which all the stochastic quantities herein are defined. We define $P$ to be the law of $X_{\Deltaup}$ or, equivalently, the common law of the independent increments $Z_k:=X_{k{\Deltaup}}-X_{(k-1){\Deltaup}}$, $k=1, \ldots, n$. Let $P_n:=n^{-1} \sum_{k=1}^n \delta_{Z_k}$ be the empirical measure of the increments and denote by ${\cal F}$ and ${\cal F}^{-1}$ the Fourier (--Plancherel) transform and its inverse acting on finite measures or on $L^2(\R)$ functions (the reader is referred to \cite{F99} for the results from Fourier analysis and convolution theory used throughout). Then for all $u\in \R$
\begin{equation*}
	\varphi(u) := \FT P(u):= \int_{\R} e^{i u x}P(dx) \qquad \mbox{and} \qquad \varphi_n(u):=\FT P_n(u):= \frac{1}{n} \sum_{k=1}^n e^{i u Z_k}
\end{equation*}
are the characteristic function of $Z$ and its empirical counterpart. A simple calculation starting from the representation for $X_{\Deltaup}$ obtained from expression \eqref{CPPFirst} shows that
\begin{equation}\label{eqLevyKhintchineCPP}
	\varphi(u) =\exp\big({\Deltaup} \left( i\gamma u + \FT \nu (u) -\lambda \right) \big),
\end{equation}
where $\nu$ is a finite measure on $\R$ referred to as the L{\'e}vy measure satisfying that for any $t\in \R$
\begin{equation*}
	\int_{-\infty}^t \nu(dx) = \lambda \, F(t) =: N(t).
\end{equation*}
The precise assumptions on $\nu$ are included in the next section. For now let us mention we assume
\begin{equation}\label{LebDecNu}
	\nu=\nu_d + \nu_{ac},
\end{equation}
where $\nu_d$ and $\nu_{ac}$ are finite measures that are discrete and absolutely continuous with respect to Lebesgue's measure, respectively. The former satisfies that
\begin{equation}\label{Nud}
	\nu_d=  \sum_{j\in \mathcal{J}} q_j \delta_{J_j}, 
\end{equation}
where $\mathcal{J}$ is countable, $q_j \in [0,\lambda]$ for all $j\in \mathcal{J}$ 
and $q:= \sum_{j\in \mathcal{J} } q_j \in [0,\lambda]$. In order to work with the actual weights of the jump measure we define for all $j\in \mathcal{J}$
\begin{equation}\label{ps}
	p_j:= \frac{q_j}{\lambda} \in [0,1] \qquad \mbox{and} \qquad p:=\frac{q}{\lambda} \in [0,1].
\end{equation}

Throughout we write $A_c\lesssim B_c$ if $A_c\le C B_c$ holds with a uniform constant $C$ in the parameter $c$ and $A_c\thicksim B_c$ if $A_c\lesssim B_c$ and $B_c\lesssim A_c$. We use the standard notation $A_n = O(B_n)$ and $A_n = o(B_n)$ to denote that $A_n/B_n$ is bounded and $A_n/B_n$ vanishes, respectively, as $n\to \infty$. We write $\xi_n=O_{\Pr}(r_n)$ to denote that $\xi_n r_n^{-1}$ is bounded in $\Pr$-probability. The natural norm of the space $L^r(\R)$, $r\in [1,\infty]$, is denoted by $\norm{\cdot}_r$ and the total variation of signed measures by $\norm{\cdot}_{TV}$.

In Section \ref{SecCLT} convergence in distribution of real-valued random variables is denoted by $\to^d$. The space of bounded real-valued functions on $\R$ equipped with the supremum norm is denoted by $\ell^\infty(\R)$ and $\to^\mathcal L$ denotes convergence in law in it (cf. \cite{D99}, p.94). Convergence in law in the product space $\R^{\N} \times \ell^\infty(\R)^2$ is denoted by $\to^{{\mathcal L}^{\times}}$ (cf. Chapter 1.4 in \cite{AvdVW96}).

\subsection{Assumptions and estimators} \label{SecSettingEstimators}

\begin{assump} \label{Assumption1}
	On the unknown L{\'e}vy measure $\nu$ we make the following assumptions:
	\begin{enumerate}[label= \arabic*]
		\item its Lebesgue decomposition is as in \eqref{LebDecNu}, satisfying that \label{enum:th:ass1}
		\begin{enumerate} 
			\item $\nu_d$ is given by \eqref{Nud}, where for some $\varepsilon>0$ fixed, $J_j=\varepsilon j $ for all $j \in \mathcal{J}:= \Z \setminus \{0\}$
			; and, \label{enum:th:ass1a}
			\item for any $s,t\in\R$ distinct, $| \int_s^t \nu_{ac}(x)\, dx| \lesssim \min\{(\log(|s - t|^{-1}))^{-\alpha}, 1\}$ for some $\alpha>4$. And,\label{enum:th:ass1b}
		\end{enumerate}
		\item $\int_{\R} \log^{\beta}\big(\max\{|x|,e\}\big) \,\nu(dx)<\infty$ for some $\beta>2$.\label{enum:th:ass2}
	\end{enumerate}
\end{assump}

Note that $\nu$ may be fully discrete or fully absolutely continuous, which is determined by the (unknown) value of $p$ defined in \eqref{ps}. We assume $\nu$ has no atom at the origin as otherwise it merges with the intensity and neither can be identified. 
The assumption of $\supp(\nu_d)\subseteq \varepsilon \times \Z \setminus\{0\}$ 
can be relaxed to a regular non-equispaced support condition and the results herein follow simply by making cosmetic changes to our proofs. However,  we present it as above for notational simplicity and refer the reader to the end of Remark 3.2.1 in \cite{C16} for more details. Assumption \ref{enum:th:ass1b} is satisfied if $\nu_{ac}\in L^{1+\epsilon}(\R)$ for some $\epsilon>0$ with much to spare: by H\"older's inequality,
\begin{equation*}
	\left|\int_{s}^t \nu_{ac}(x)\, dx \right| \leq \min\left\{\norm{\nu_{ac}}_{1+\epsilon} \, |s-t|^{\epsilon/(1+\epsilon)},\lambda\right\} \quad \mbox{ for any } s,t \in \R.
\end{equation*}

A kernel function $K$ features in the estimators and we assume it is symmetric and it satisfies
\begin{equation}\label{conditionsK}
	\int_{\R}K(x) \, dx=1, \quad \supp({\cal F}K)\subseteq[-1,1] \quad \mbox{and}  \quad \abs{K(x)}\lesssim (1+\abs{x})^{-\eta}  \text{ for some } \eta>2.
	\vspace{-0.1cm}
\end{equation}
Therefore, the functions $K_{h_n}:= h_n^{-1} K(\cdot/ h_n)$, where $h_n$ is referred to as the bandwidth, are continuous, have Fourier transform $\FT K_{h_n}$ supported in $[-h_n^{-1},h_n^{-1}]$ and provide an approximation to the identity operator as $h_n\to 0$ when $n\to \infty$. 

We now heuristically motivate the form of our estimators using \eqref{eqLevyKhintchineCPP}. Taking logarithms and, formally, inverse Fourier transforms, and integrating against a regular enough function $f: \R \to \R$,
\begin{align}\label{eqHeuristicsEstimators}
	\frac{1}{{\Deltaup}} \! \int_{\R} \! f \, \FII{\Log (\varphi) } \!=\! \int_{\R} \! f \, d\big[ \!- \!\gamma  ( \delta_0)' \!-\! \lambda  \delta_0 \!+\! \nu \big] \!=\! \gamma f'(0) \!  - \! \lambda f(0) \!+\!\!\!\! \sum_{j\in  \Z\setminus\{0\}} \!\!\!\!q_j   f(J_j) \!+\! \int_{\R} \! f \, d\nu_{ac}.
	\vspace{-0.3cm}
\end{align}
Here, $\Log(\varphi)$ denotes the distinguished logarithm of $\varphi$, i.e. the unique continuous function satisfying $\exp \!\big(\!\Log(\varphi)(u)\big)\!=\!\varphi(u)$ for all $u \!\in\! \R$ (cf. Theorem 7.6.2 in \cite{C01} for its construction). 
Note that all of the summands on the right hand side are unknown, whilst $\varphi$ on the left hand side can be estimated by $\varphi_n$. Thus, the idea is first to write each parameter solely in terms of $\varphi$ by choosing $f$ such that the rest of the summands are zero or vanish asymptotically, and second to justify the substitution by $\varphi_n$. 
In the proofs, the role of $\delta_0$ is played by $K_{h_n}$ so, despite $(\delta_0)'$ is to be understood in a distributional sense, $f$ does not have to be smooth in practice; it suffices that it is differentiable at the origin and uniformly bounded. Then, in view of the last display, due to Assumption \ref{enum:th:ass1b} and taking $\zeta_n\to 0$ exponentially and sufficiently fast depending on $\alpha$,
\begin{align}\label{identitygammalambda}
	\gamma = \lim_{n\to \infty} \frac{1}{{\Deltaup}}  \int_{\R} x \1_{  |x|< \zeta_n } \FII{\Log (\varphi) }(dx), \quad \lambda= - \lim_{n\to \infty} \frac{1}{{\Deltaup}}  \int_{\R} \1_{ |x|< \zeta_n } \FII{\Log (\varphi) }(dx), 
\end{align}
\vspace{-0.3cm}
\begin{align}\label{identityqjs}
	q_j= \lim_{n\to \infty} \frac{1}{{\Deltaup}}  \int_{\R} \1_{ |x-J_j|< \zeta_n } \FII{\Log (\varphi) }(dx), \quad j \in \Z \setminus \{0\},
	\vspace{-0.3cm}
\end{align}
and
\begin{equation*}\label{identitynuac}
	\vspace{-0.1cm}
	\int_{-\infty}^t \nu_{ac}= \lim_{n\to \infty} \frac{1}{{\Deltaup}}  \int_{-\infty}^t \1_{|x| \geq  \zeta_n } \Big( 1  - \sum_{ j\in \Z \setminus \{0\} } \1_{   |x-J_j|< \zeta_n   }  \Big)  \FII{\Log (\varphi)}(dx), \quad t\in \R.
	\vspace{-0.1cm}
\end{equation*}
Noticing that $N(t)=\int_{-\infty}^t \nu$ and that $\lambda$ is the mass of $\nu$, we also have that
\begin{equation*}\label{identitynulambda}
	N(t) =  \lim_{n\to \infty} \frac{1}{{\Deltaup}} \int_{-\infty}^t \! \1_{|x| \geq \zeta_n }  \FII{\Log (\varphi)}(dx) \quad \mbox{and} \quad \lambda= \lim_{n\to \infty} \frac{1}{{\Deltaup}} \int_{\R}  \1_{|x| \geq \zeta_n }  \FII{\Log (\varphi)}(dx).
\end{equation*}
This is one of the advantages of the spectral approach: it is possible to untangle and isolate the influence that each parameter has on the characteristic function of the observations and to invert it. In turn, it also shows that the model lacks no identifiability. In particular, the truncation around the origin in the last three identities has an intuitive meaning: it kills the influence of $\gamma$ and $\lambda$ coming from the observations $Z_k$ that carry empty sums; it therefore corresponds to the practice found in the rest of the literature (in which $\gamma=0$) of throwing away the zero observations; remarkably, it does this automatically and, more importantly, it only throws away those coming from empty sums as shown by \cite{C16} in Section 2.4.1. Hence, this property is key in dealing with the novel setting we consider here in which $\nu_d$ may result in jumps that cancel each other and hide in the observations since, for the user, these are indistinguishable from the former. 

The next steps are to introduce $\varphi_n$ in place of $\varphi$ in the identities above, to give the final expressions for the estimators and to justify in what sense they are well-defined. We start with $\lambda$ and remark that, as shown in \cite{C16}, the estimators of it constructed from the last identity and from \eqref{identitygammalambda} are equal up to terms vanishing exponentially fast as $n\to \infty$ and all the conclusions herein do not depend on which is used. We construct it from the last display and propose
\begin{equation} \label{eqLambdahat}
	\hat \lambda_{n} := \frac{1}{{\Deltaup}} \int_{\R} f^{(\lambda)}_n(x) \, \FII{ \Log(\varphi_n) \FT K_{h_n}} (x) \, dx, \qquad f^{(\lambda)}_n(x):=\1_{  \varepsilon_n\leq |x|\leq H_n  },
\end{equation}
where $\varepsilon_n \to 0$ and $H_n \to \infty$ are determined later. For now we remark that we will need $h_n=o(\varepsilon_n)=o(n^{-\vartheta})$ for any $\vartheta>0$ so that $K_{h_n}$ converges to $\delta_0$ sufficiently fast for the arguments leading to the identities above to hold. In the implementation of $\hat \lambda_{n}$ and the estimators below, the truncation of the tails of the integral is natural, and, in the proofs, it allows us to control the errors under Assumption \ref{enum:th:ass2} instead of under a finite polynomial moment condition. The last display, and the rest of the estimators, is well-defined in sets of $\Pr$-probability approaching $1$ as $n\to \infty$ for the following reasons: in our proofs we show that, under Assumption \ref{enum:th:ass2} and if $h_n^{-1}$ does not grow too fast, $\sup_{u\in[-h_n^{-1},h_n^{-1}]} |\varphi_n(u) -\varphi(u)| \to 0$ in $\Pr$-probability; in view of \eqref{eqLevyKhintchineCPP},
\begin{equation}\label{eq:boundvarphi}
	0< e^{-2 \lambda {\Deltaup} }  \leq |\varphi(u)| \leq 1 \quad \mbox{ for all } u\in \R,
\end{equation}
so, in the above-mentioned sets, 
$\Log(\varphi_n) \FT K_{h_n}$ is well-defined and in $L^{r}(\R)$ for any $r\!\in\![1,\infty]$; furthermore, 
$f^{(\lambda)}_n \! \in \! L^{r}(\R)$ for any $r\!\in\![1,\infty]$, and by basic Fourier analysis arguments the estimator is well-defined in such sets. 
To estimate the drift we propose
\begin{equation*}\label{eqMuhat}
	\hat \gamma_n :=  \frac{1}{ {\Deltaup}} \int_{\R} f^{(\gamma)}_n (x) \, \FII{ \Log(\varphi_n) \FT K_{h_n}} (x) \, dx, \qquad f^{(\gamma)}_n (x) :=c^{-1} x \1_{  |x|< h_n },
\end{equation*} 
where $c:= 2 \, (\int_{[0,1]}K(x) \, dx - K(1) )$ is assumed to be distinct from zero. In view of \eqref{identitygammalambda}, here we have chosen $\zeta_n=h_n$, which is the speed at which $K_{h_n}$ converges to $\delta_0$. This equal speed results in the appearance of the unusual constant $c$ and, through the $x$ term in $f^{(\gamma)}_n$, in the rate of convergence $h_n/\sqrt{n}$ which can therefore not be improved upon as otherwise $K_{h_n}$ does not converge to $\delta_0$ fast enough. For any $j\in \Z \setminus\{0\}$ we estimate the $j$-th weight of $\nu_d$ by
\begin{equation*}\label{eqqhat}
	\hat q_{j,n} :=  \frac{1}{{\Deltaup}} \int_{\R} f^{(q_j)}_n(x) \, \FII{ \Log(\varphi_n) \FT K_{h_n}} (x) \, dx, \qquad f^{(q_j)}_n(x):=\1_{ |x-J_j|< \varepsilon_n },
\end{equation*} 
and thus the parameter $q$ can be estimated by
\begin{align*}
	\hat q_n := \sum_{\substack{|j| \leq \widetilde H_n/\varepsilon \\ j \neq 0}} \hat q_{j,n} = \frac{1}{{\Deltaup}} \int_{\R} f^{(q)}_n(x) \, \FII{ \Log(\varphi_n) \FT K_{h_n}} (x) \, dx, \quad f^{(q)}_n(x):=\sum_{\substack{ |j| \leq \widetilde H_n/\varepsilon \\ j \neq 0}} \1_{  |x-J_j| < \varepsilon_n   },
	\vspace{-0.2cm}
\end{align*}
where $\widetilde{H}_n=o(H_n)\to \infty$ for technical reasons. To estimate $N(t)$ it is tempting to take
\begin{align}\label{eqNhatAC}
	\frac{1}{{\Deltaup}} \int_{-\infty}^t   \1_{  \varepsilon_n\leq |x|\leq H_n  } \, \FII{  \Log(\varphi_n) \FT K_{h_n }}(x) \,  dx \quad \mbox{for any } t\in\R.
\end{align}
Due to the regularisation induced by $K$, this quantity is continuous in $t$ in the aforementioned sets. 
However, if $\nu_d$ is not null, the function $N$ is discontinuous. Hence, and given that we seek a limit theorem under the uniform norm, we need to make the estimator discontinuous. The general idea we propose is to estimate $N(t)$ by the last display everywhere except at increasingly small intervals around each of the potential jumps; in these, the estimate is kept constant as $t$ grows until it equals the potential atom; at that point, the estimate of its mass (equal to the estimated mass in the interval) is added and the estimator remains constant until the end of the interval. Assuming $|H_n- J_j|>\varepsilon_n$ for all $j\in \Z\setminus\{0\}$ and $\varepsilon_n<\varepsilon/2$ without loss of generality, 
we therefore propose the estimator
\begin{align*} 
	\widehat{N}_n (t) := \frac{1}{\Deltaup} \int_{\R}  f^{(N)}_{t,n} (x) \, \FI \big[ \Log(\varphi_n) \FT K_{h_n} \big] (x) \, dx,
\end{align*}
where, defining $J_0:=0$ for notational simplicity,
\begin{equation}\label{defftn}
	f^{(N)}_{t,n}= \1_{[-H_n, H_n] \setminus (-\varepsilon_n, \varepsilon_n)} \times \left\{  \begin{array}{lll}
		\1_{(-\infty, t]} & \mbox{if} \quad |t-J_j|>\varepsilon_n &  \mbox{for all } j \in  \Z, \\
		\1_{(-\infty, J_j-\varepsilon_n]}  &  \mbox{if}  \quad  J_j-\varepsilon_n \leq t < J_j  &  \mbox{for some } j \in  \Z, \\
		\1_{(-\infty, J_j+\varepsilon_n]}  &  \mbox{if}  \quad J_j \leq t \leq J_j+\varepsilon_n  &  \mbox{for some } j \in  \Z.
	\end{array}
	\right.
\end{equation}
This can be alternatively written as 
\begin{align} \label{eqNhat}
	\widehat{N}_n (t) \!:=\! \frac{1}{\Deltaup}\!  \int_{-\infty}^t \!\!\!\! \1_{   |x|\leq H_n   } \!\Big(  1 \!-\! \sum_{ j\in \Z} \!\1_{   |x-J_j|< \varepsilon_n   }\! \Big) \FI \big[ \Log(\varphi_n) \FT K_{h_n} \big] \!(x)  dx + \!\!\!\! \sum_{\substack{|j| \leq H_n/\varepsilon \\ j\leq t/\varepsilon, \, j \neq 0} } \!\!\!\! \hat q_{j,n},
	\vspace{-0.3cm}
\end{align}
where the first term is estimating the cumulative function of $\nu_{ac}$ and the second estimates the cumulative discrete component by adding a new estimate exactly at the value of each potential atom. It follows from our proofs that if one of the two components is zero, the corresponding term is asymptotically negligible uniformly in $t$. Thus, that term can be discarded if there is such a priori knowledge and the conclusions herein still follow under the same assumptions; nonetheless, if $\nu_d$ is null, it is simpler to use \eqref{eqNhatAC} directly. 
For further simplifications of the implementation of $\widehat{N}_n$, see Sections 4.3 and 4.4 in \cite{C16}, and for additional remarks about the transformation of \eqref{eqNhatAC} given by \eqref{defftn}, see Remark 3.2.1 in \cite{C16}. Finally, we define the rest of the estimators as
\begin{align}\label{eqRescEst}
	\widehat{F}_n  := \hat \lambda_n^{-1} \widehat{N}_n, \qquad \hat p_{j,n} := \hat \lambda_n^{-1} \hat q_{j,n}, \, \, \, j \in \Z\setminus\{0\}, \quad \mbox{ and } \quad \hat p_{n} := \hat \lambda_n^{-1} \hat q_{n}.
\end{align}
Note that, for $t$ negative enough $\widehat{F}_n(t)=0$ and, due to the expression for $\hat \lambda_n$ and the first one for $\widehat{N}_n$, $\widehat{F}_n(t)=1$ for $t$ large enough. We remark that, despite the appearance of distinguished logarithms and Fourier transforms, all the estimators are real valued for any $n$.

\subsection{Central limit theorems} \label{SecCLT}

Our main result concerns joint estimation of the parameters introduced above. However, to ease the exposition we first specialise it to each of them and interpret the results separately. 

\subsubsection{Marginal convergence of the estimators}

\begin{prop}\label{Proplambdamu}
	Suppose Assumptions \ref{enum:th:ass1} and \ref{enum:th:ass2} are satisfied for some $\alpha\!>\!4$ and $\beta\!>\!2$, and that $K\!$ satisfies \eqref{conditionsK}. Let $h_n \!\sim\! \exp(-n^{\vartheta_h})$,  $\varepsilon_n\! \sim\! \exp(-n^{\vartheta_\varepsilon})$ and $H_n \!\sim \!\exp(n^{\vartheta_H})$, where $1/\alpha\!<\!2\vartheta_\varepsilon\! \leq \!\vartheta_h \!<\! 1/4$ and $1/(2\beta) \!\leq \!\vartheta_H\!<\!\vartheta_h$. Then, under the notation at the end of Section \ref{SecDefNot}, we have that
	\begin{equation*}
		\sqrt{n} \, \Big( \hat \lambda_n - \lambda \Big) \to^d N\big(0, \sigma^2_\lambda \big) \quad \mbox{and} \quad \sqrt{n} \, h_n^{-1} \big( \hat \gamma_n - \gamma \big) \to^{\Pr} 0
	\end{equation*}
	as $n\to \infty$, where in what follows $N(0,\sigma^2)$ denotes a zero-mean normal distribution with variance $\sigma^2$ and, writing $f^{(\lambda)}:= \1_{\R\setminus \{0\}}$ and $\varphi^{-1}=1/\varphi$ throughout,
	\begin{align*}
		\sigma^2_\lambda:=  \frac{1}{{\Deltaup}^2}  \int_{\R}  \Big(  f^{(\lambda)} \! \ast \! {\cal F}^{-1} \left[ \varphi^{-1}(-\cdot) \right] (x)\Big)^2  P(dx).
	\end{align*}
\end{prop}

We first show the finiteness of $\sigma^2_\lambda$ for a more general form of it so that the justification extends to similar expressions encountered later in this section. Notice that
\begin{equation}\label{eqFIPhim1nP}
	P=e^{-\lambda \Deltaup} \,  \delta_{\gamma \Deltaup} \ast  \sum_{k=0}^{\infty} \nu^{\ast k} \frac{{\Deltaup}^k}{k!}  \quad \mbox{ and } \quad \FII{\varphi^{-1}(-\cdot)} = e^{\lambda {\Deltaup}} \,  \delta_{\gamma \Deltaup} \ast \sum_{k=0}^\infty \bar \nu^{\ast k} \frac{(-{\Deltaup})^{k}}{k!},
\end{equation}
where $\bar{\nu}(A):=\nu(-A)$ for all $A \subseteq \R$ Borel measurable, $\bar \nu^{\ast k}$ denotes the $(k-1)$-fold convolution of $\bar{\nu}$ and $\bar \nu^{\ast 0}=\delta_0$ by convention. Therefore, $\FII{\varphi^{-1}(-\cdot)}$ is a finite signed measure and we remark that its series representation, and the fact that it has mass equal to $1$, follows by similar arguments to those used to justify the well-known representation of $P$ above (see Remark 27.3 in \cite{S99}). Consequently, for any bounded functions $f_1,f_2:\R\to \R$,
\begin{equation}\label{eqFinVar}
	\int_{\R} \Big( f_1 \ast {\cal F}^{-1}  \left[ \varphi^{-1}(-\cdot) \right] (x) \Big) \Big( f_2  \ast {\cal F}^{-1} \left[ \varphi^{-1}(-\cdot) \right] (x) \Big) \, P(dx) < \infty.
\end{equation}
Furthermore, if ${\cal F}^{-1}  \left[ \varphi_0^{-1}(-\cdot) \right]$ and $P_0$ are the measures above when $\gamma\!=\!0$, the last display is
\begin{align*}
	\int_{\R}  \Big( f_1& \ast  {\cal F}^{-1}  \left[ \varphi_0^{-1}(-\cdot) \right] \ast \delta_{\gamma \Deltaup} (x) \Big) \Big( f_2  \ast {\cal F}^{-1} \left[ \varphi_0^{-1}(-\cdot) \right] \ast \delta_{\gamma \Deltaup} (x)\Big)  P_0 \ast \delta_{\gamma \Deltaup} (dx)\\
	& = \int_{\R} \Big( f_1 \ast {\cal F}^{-1}  \left[ \varphi_0^{-1}(-\cdot) \right] (x) \Big) \Big( f_2  \ast {\cal F}^{-1} \left[ \varphi_0^{-1}(-\cdot) \right] (x) \Big) P_0(dx).
\end{align*}
Therefore, in all the limiting quantities below the drift does not play any role and in some remarks we thus assume $\gamma\!=\!0$. This agrees with the observation  in \cite{T15} that $\gamma$ has no influence on the information lower bounds of the problem and also with the null limiting quantity for $\gamma$ in Proposition \ref{Proplambdamu}. The latter and the $h_n/\sqrt{n}$-rate may seem surprising at first. When the path of the compound Poisson process at hand is perfectly observed, $\gamma$ can be \textit{exactly} computed as the slope of the process in the time intervals when it does not jump and thus has no influence on the estimation of the other parameters. Hence, these phenomena can be understood as a consequence and as further evidence of the marginal ill-posedness of our inverse problem.

In \cite{T15}, Trabs showed that if $\nu$ is absolutely continuous, the Cram{\'e}r--Rao bound for the limiting variance of a $1/\sqrt{n}$-consistent estimator of a functional $\int_{\R} f d\nu$ is \eqref{eqFinVar} with $f_1=f_2=f \1_{\R\setminus\{0\}}$. The truncation induced by the indicator is crucial: by basic Fourier analysis
\begin{equation*}
	\int_{\R} f \ast {\cal F}^{-1} \left[ \varphi^{-1}(-\cdot) \right] (x)  \, P(dx)=f(0) \quad \mbox{ for any $f$ uniformly bounded,}
\end{equation*}
and hence $f \1_{\R\setminus\{0\}}$ is in the range of the adjoint score operator of the problem. Recalling that $\lambda=\int_{\R} \nu$, we therefore achieve the lower bound in a larger model than that in \cite{T15} and conclude that the expression of the bound does not change in our model and that our estimator is efficient. This variance can be analysed further: after some algebra using the representations in \eqref{eqFIPhim1nP},
\begin{align}\label{decsigmalambda}
	\sigma^2_\lambda \!=\! \frac{P(\{0\}) \big( {\cal F}^{-1}  \left[ \varphi^{-1}(-\cdot) \right] (\{0\}) \big)^2\!-\!1}{{\Deltaup}^2} + \frac{1}{{\Deltaup}^2} \!\!\sum_{j\in \Z \setminus \{0\}} \!\! P(\{J_j\}) \Big( {\cal F}^{-1}  \left[ \varphi^{-1}(-\cdot) \right] (\{J_j\}) \Big)^2\!.
\end{align}
The estimator $\hat \lambda_n$ is able to decompound the observations $Z_k$ that are $\gamma \Deltaup$ as a result of exact cancellations between jumps arising from $\nu_d$ and the last display explicitly shows how this difficulty features in the complexity of estimating $\lambda$. Indeed, in the cases considered by existing literature, in which $\nu$ is absolutely continuous ($p\!=\!0$) or supported only in the positive real line, the infinite sum above is zero and $\sigma^2_\lambda$,  simplifies to $\sigma^2_\lambda = (e^{\lambda \Deltaup} - 1)/{\Deltaup}^2$.
This coincides with the asymptotic variance of the naive estimator $- \Deltaup^{-1} \log \left( n^{-1}\sum_{k=1}^n \1_{\{0\}} (Z_k) \right)$, which is generally used in these simpler settings. As shown in  Section 2.4.1 in \cite{C16}, there the two estimators are equal up to exponentially negligible terms. However, in our general setting the naive one is not even consistent due to its lack of ability to decompound. From the simplification of $\sigma^{2}_\lambda$ we notice that when $\lambda {\Deltaup}$, the expected number of jumps in a ${\Deltaup}$-observation window, is small, 
\begin{equation}\label{eqlambdahighfreq}
	\sqrt{n \lambda {\Deltaup}} \, \Big( \hat \lambda_n - \lambda \Big)  \approx N(0,\lambda^2).
\end{equation}
The limiting variance $\lambda^2$ is the inverse of the Fisher information for the problem of estimating an exponential distribution with parameter $\lambda$ from independent observations of it. Therefore, 
it agrees with the fact that in the continuous observations regime 
the intensity can be estimated from the interarrival times using the maximum likelihood estimator. The last display 
holds in our general setting in view of the following lemma.

\begin{lem}\label{lemmaAsympSigma}
	Let $f_1,f_2$ be bounded functions and let $\FII{\varphi^{-1}(-\cdot)} $ and $P$ be as in \eqref{eqFIPhim1nP}. Then,
	\begin{align*}
		\vspace{-0.3cm}
		& \! \! \int_{\R} \Big(f_1 \ast  {\cal F}^{-1} \left[ \varphi^{-1}(-\cdot) \right] (x)  \Big) \Big( f_2  \ast {\cal F}^{-1} \left[ \varphi^{-1}(-\cdot) \right] (x)\Big)\, P(dx) \notag \\
		& = \! f_1(0)  f_2(0) + \Deltaup \! \left( \int_{\R} \!f_1 f_2\, \nu - f_1(0) \, f_2 \ast \bar \nu \, (0) - f_2(0) \, f_1 \ast \bar \nu \, (0) + \lambda \, f_1(0)  f_2(0) \! \right) \!+ \! O\big((\lambda \Deltaup)^2\big).
		\vspace{-0.3cm}
	\end{align*}
	In particular, if $f_1(0)=f_2(0)=0$, the last display is $ \lambda \Deltaup \int_{\R} f_1(x) f_2(x) \, F(dx)  + O\big((\lambda \Deltaup)^2\big).$
\end{lem}

The following result deals with estimation of the mass of each  potential atom in $N$ and $F$. Its conclusions already hint at some differences between estimating these two functions.

\begin{prop}\label{Proppsqs}
	Under the assumptions and notation of Proposition \ref{Proplambdamu} we have that
	\begin{equation*}
		\sqrt{n} \, \big( \hat q_{j,n} - q_{j} \big) \to^d N\big(0, \sigma^2_{q_j} \big) \quad \mbox{ and } \quad \sqrt{n} \, \big( \hat p_{j,n} - p_{j} \big) \to^d N\big(0, \sigma^2_{p_j} \big)
	\end{equation*}
	as $n\to \infty$ for any $j\in \Z\setminus\{0\}$, where, defining $f^{(q_j)} := \1_{J_j}$ and $f^{(p_j)} := \lambda^{-1} \big(f^{(q_j)} -p_j f^{(\lambda)} \big)$,
	\begin{align*}
		\sigma^2_{q_j}\! :=\! \frac{1}{{\Deltaup}^2 }\!\! \int_{\R}\!\! \Big(\! f^{(q_j)} \! \ast \! {\cal F}^{-1}\! \left[ \varphi^{-1}(-\cdot) \right] (x)\!\Big)^2\!\! P(dx) \, \, \mbox{ and } \, \,
		\sigma^2_{p_j}\! :=\! \frac{1}{{\Deltaup}^2}\!\! \int_{\R}\!\! \Big(\! f^{(p_j)}\! \ast\! {\cal F}^{-1}\! \left[ \varphi^{-1}(-\cdot) \right] (x)\!\Big)^2 \!\! P(dx). 
	\end{align*}
\end{prop}

Similar to the expression for $\sigma_{\lambda}^2$ above, we have that for any $j\in \Z \setminus\{0\}$,
\begin{equation*}
	\sigma^2_{q_j} = \frac{1}{{\Deltaup}^2} \sum_{l\in \Z} P(\{J_l\}) \Big( {\cal F}^{-1}  \left[ \varphi^{-1}(-\cdot) \right] (\{J_l-J_j\}) \Big)^2, 
\end{equation*}
where we recall that $J_0:=0$ for notational simplicity, and
\begin{equation*}
	\sigma^2_{p_j} = \frac{1}{{\Deltaup}^2\lambda^2} \bigg( \sum_{l\in \Z} P(\{J_l\}) \Big( {\cal F}^{-1}  \left[ \varphi^{-1}(-\cdot) \right] (\{J_l-J_j\}) + p_j {\cal F}^{-1}  \left[ \varphi^{-1}(-\cdot) \right] (\{J_l\}) \Big)^2-p_j^2 \bigg).
\end{equation*}
Both expressions are zero when $\nu$ is absolutely continuous ($p=0$) because the probability of one or more increments $Z_k$ taking values in $\varepsilon \times \Z \setminus \{0\}$ is zero and there are no discrete jumps to be decompounded. 
We can gain more intuition when $\lambda \Deltaup$ is sufficiently small. By Lemma \ref{lemmaAsympSigma},
\begin{equation*}
	\sqrt{n \lambda {\Deltaup}} \, \big( \hat q_{j,n} - q_j \big)  \approx N\big(0, \lambda^2 p_j\big) \quad \mbox{and} \quad \sqrt{n \lambda {\Deltaup}} \, \big( \hat p_{j,n} - p_j \big)  \approx N\big(0, p_j(1-p_j)\big).
\end{equation*} 
In this regime most if not all of the jumps are observed directly so little or no decompounding is needed and the asymptotic variances do not depend on weights other than the one under consideration.  The dependence on $p_j$ of the first limiting variance hints at the Brownian motion type of limit already mentioned in the introduction: it shows how estimating $q_j$ carries the uncertainty of not knowing $\lambda$ agreeing with \eqref{eqlambdahighfreq}. Dividing its estimate by that of $\lambda$ to estimate $p_j$ removes this variability and the second limiting variance agrees with Donsker's theorem.


We now address the estimation of $N$ and $F$, the main contribution of this work.

\begin{thm}\label{ThmNF}
	Under the assumptions of Proposition \ref{Proplambdamu} and the notation of Section \ref{SecDefNot},
	\begin{equation*}
		\sqrt{n} \, \Big(\widehat{N}_n - N \Big)  \to^{\mathcal{L}} \mathbb B^N \, \, \mbox{ in } \ell^{\infty}(\R) \quad \mbox{ and } \quad \sqrt{n} \,  \Big(\widehat{F}_n - F \Big)  \to^{\mathcal{L}} \mathbb G^F \, \, \mbox{ in } \ell^{\infty}(\R)
	\end{equation*}
	as $n\to \infty$, where, defining for any $t\in\R$
	\begin{equation*}
		f^{(N)}_t \!:= \1_{(-\infty, t]}  \1_{\R \setminus \{0\}} \quad \mbox{ and } \quad f^{(F)}_t \! := \lambda^{-1} \big( f^{(N)}_t  - F(t) f^{(\lambda)} \big)=\lambda^{-1} \big( \1_{(-\infty, t]}  - F(t) \big) \1_{\R \setminus \{0\}}, 
	\end{equation*}
	$\mathbb B^N\!$ and $\mathbb G^F\!$ are tight centred Gaussian Borel random variables in $\ell^\infty(\R)$ with covariance structures 
	\begin{align}\label{eqcovN}
		\Sigma_{s,t}^{N}:= \frac{1}{{\Deltaup}^2}  \int_{\R}  \Big( f^{(N)}_s \ast {\cal F}^{-1} \left[ \varphi^{-1}(-\cdot) \right] (x) \Big) \Big( f^{(N)}_t \ast {\cal F}^{-1} \left[ \varphi^{-1}(-\cdot) \right] (x) \Big)  P(dx), \quad s,t\in \R,
		\vspace{-0.3cm}
	\end{align}
	and
	\begin{align}\label{eqcovF}
		\vspace{-0.3cm}
		\Sigma_{s,t}^{F}:= \frac{1}{{\Deltaup}^2}  \int_{\R}  \Big( f^{(F)}_s & \ast {\cal F}^{-1} \left[ \varphi^{-1}(-\cdot) \right] (x) \Big) \Big( f^{(F)}_t \ast {\cal F}^{-1} \left[ \varphi^{-1}(-\cdot) \right] (x) \Big)  P(dx), \quad s,t\in \R.
	\end{align}
\end{thm}

\subsubsection{Discussion of efficiency of Theorem \ref{ThmNF}}

Nickl and Rei{\ss} \cite{NR12} constructed an efficient estimator and showed a central limit theorem for the generalised L{\'e}vy distribution $N^{G}(t):=N(t)$ if $t<0$ and $N^{G}(t):=\int_{t}^{\infty}d\nu$ if $t>0$. Indeed, when $t<0$ the variance of $\mathbb B^N$ coincides with that of their limiting process; on the positive real line they do not agree simply because the quantities being estimated are different. As they remark and we further illustrate with our result on $F$, the limiting covariances are intuitively appealing when compared to the classical Donsker theorem. Furthermore, they agree with those in \cite{ST12}, where a functional central limit theorem for the classical deconvolution problem is shown.

The truncation at the origin induced by the indicator $\1_{\R\setminus\{0\}}$ in $f^{(N)}_t$ and $f^{(F)}_t$ guarantees that our estimators are efficient\footnote{In personal communication with M. Trabs, it was confirmed that the function $\1_{(-\infty,t]}$ in Corollary 4.5 in \cite{T15} should be $f^{(N)}_t$ for their results to be correct, and $\tilde \chi_\nu$, defined after the corollary, should equal $\lambda f_t^{(F)}$. We also confirmed the covariance of the limiting process in the general case considered in \cite{BG03} does not coincide with $\Sigma_{s,t}^{F}$.}: the conclusion for $\widehat{N}_n$ follows by the same arguments used right before \eqref{decsigmalambda} to conclude efficiency of $\hat{\lambda}_n$; and the justification of \eqref{eqcovF} being optimal requires some clarification which we believe gives further understanding of the efficiency of estimators of $F$, not addressed in \cite{T15}. There, it is shown that the lower bound for the asymptotic covariance when estimating $N$ with $\lambda$ known is $\lambda^2 \Sigma_{s,t}^{F}$ and, consequently, the lower bound for estimating $F$ in that setting must be $\Sigma_{s,t}^{F}$. Yet, the estimation problem when $\lambda$ is unknown is harder and $\Sigma_{s,t}^{F}$ must also be a lower bound in this more complex setting. Furthermore, because we are considering an even more involved framework in which $\nu$ may have a discrete component that results in exact cancellations, the Cram{\'e}r--Rao bound is $\Sigma_{s,t}^{F}$ and our estimator is efficient. Two conclusions can be drawn from this analysis: when $\lambda$ is unknown, the problems of estimating $N$ and $F$ are different from an information-theoretic point of view; and, when knowledge of $\lambda$ is available, our estimators in \eqref{eqRescEst} need not be modified. In particular, $\hat \lambda_n$ should not be substituted by $\lambda$ as efficiency is lost otherwise. In this setting $N$ should be estimated by $\lambda \widehat{F}_n$ and not by $\widehat{N}_n$. This way we input our knowledge of $\lambda$ and in view of the remarks we just made the estimator attains the information-theoretic lower bound developed in \cite{T15}. Clearly, $\widehat{N}_n$ would not attain it in this setting. Knowledge of $\gamma$ does not change any of the information bounds and thus $N$ contains all the information about the underlying process.

The main difference between \eqref{eqcovN} and \eqref{eqcovF} is the appearance of a negative extra summand in the latter. It implies that the first limiting process is of Brownian motion type whilst the second is of Brownian bridge type. This is best seen when $\lambda \Deltaup$ is small since, by Lemma \ref{lemmaAsympSigma}, 
\begin{equation*}
	\lambda \Deltaup \Sigma_{s,t}^{N} \approx  \lambda^2 F (\min\{ s,t\} ) \quad \mbox{ and } \quad  \lambda \Deltaup \Sigma_{s,t}^{F} \approx  F (\min\{ s,t\} )  - F(s) F(t).
\end{equation*}
The first approximation agrees with the results developed by Nickl et al. in \cite{NRST16}, and 
the second with Donsker's theorem because it implies that for $n$ large and $\lambda{\Deltaup}$ small, $\sqrt{\lambda  {\Deltaup}  n} \,  \big(\widehat{F}_n - F \big)$ is approximately an $F$-Brownian bridge. 
We emphasise that last display is another consequence of the appearance of $\1_{\R\setminus\{0\}}$ in $f^{(N)}_t$ and $f^{(F)}_t$, as it annihilates the leading term in Lemma \ref{lemmaAsympSigma}. 
In view of the remark in the paragraph right before \eqref{eqLambdahat}, it comes from our way of (automatically) discarding the observations with no jumps, which therefore only enter the estimators through $n$ or, equivalently, through the only useful information they contain: the number of them in a sample of size $n$. 
For the case of $\lambda\Deltaup$ general, the difference between $\mathbb B^N$ and $\mathbb G^F$ can also be concluded from their covariances: since all measures in $\Sigma_{t,t}^{N}$ and $\Sigma_{t,t}^{F}$ are finite, it can be argued that the limit of the variances as $t\to \pm \infty$ can be transferred to $f_t^{(N)}$ and $f_t^{(F)}$, and
\begin{equation*}
	\lim_{t\to -\infty} \Sigma_{t,t}^{N}=\lim_{t\to \pm \infty} \Sigma_{t,t}^{F}=0 \quad \mbox{and} \quad \lim_{t\to \infty} \Sigma_{t,t}^{N}=\sigma_{\lambda}^2. 
\end{equation*}
This has an intuitive explanation: from the expressions for $\widehat{N}_n(t)$ and $\hat \lambda_n$, $\lim_{t\to \infty} \widehat{N}_n(t) = \hat \lambda_n$ and, by Proposition \ref{Proplambdamu}, its difference with $\lim_{t\to \infty} N(t)=\lambda$ multiplied by $\sqrt{n}$ converges to a non-degenerate centred normal distribution as $n\to \infty$. This approximate normality as $t\to \infty$, and thus the stochasticity of $\widehat{N}_n$ arising from the lack of knowledge of $\lambda$, disappears when dividing $\widehat{N}_n$ by $\hat \lambda_n$. Hence, $\sqrt{n} \big(\widehat{F}_n-F\big)$ and its limit process are tied down to zero at infinity transforming the previous Brownian motion type of process into a Brownian bridge type of process. 

These ideas allow us to show that the general estimator of $F$ in \cite{BG03} is not efficient. It is constructed assuming knowledge of $\lambda$ and directly inputting its value. Recall that in \cite{BG03} they assume $\supp(F)\subseteq\R^{+}$. Taking $t\geq 0$, the variance of their limiting process is given by 
\begin{align*}
	\frac{1}{(\lambda \Deltaup)^2} \!\int_0^t \!\int_0^t \!\Big( G_0\big( (t\!-\!x) \!\wedge\! (t\!-\!y) \big) \!-\! G_0(t\!-\!x) G_0(t\!-\!y) \Big) \FII{\varphi^{-1}}(dx) \, \FII{\varphi^{-1}}(dy),
\end{align*}
where $G_0(x):=\int_{(0,x]} P$. This is zero at $t=0$ and, noting that $\lim_{x\to\infty}G_0(x)=1-P(\{0\})=1-e^{-\lambda \Deltaup}$ and recalling that $\FII{\varphi^{-1}(-\cdot)}$ has mass $1$, its limit as $t\to \infty$ is $e^{-\lambda \Deltaup} \left(1-e^{-\lambda \Deltaup}\right)/(\lambda \Deltaup)^2$. This is strictly positive and disagrees with the behaviour of $\Sigma_{t,t}^{F}$ as $t\to \infty$, so their estimator is not efficient. Indeed, the limiting process is of Brownian motion type and, as in our estimator of $F$, inputting the value of $\lambda$ when it is known leads to a suboptimal limiting process.

\subsubsection{Joint convergence of the estimators}

All the results included so far are particular cases of the following, which establishes joint convergence of all the estimators and is our main result. Prior to stating it let us introduce some notation. For any $\boldsymbol L\in \R^4 \times \Big(\R^{\Z\setminus\{0\}}\Big)^2 \times \big(\ell^{\infty}(\R)\big)^2 \simeq \R^{\N}\times \ell^{\infty}(\R)^2$ we denote its first four coordinates by $\boldsymbol L_{\lambda}, \boldsymbol L_{\gamma}, \boldsymbol L_{q}, \boldsymbol L_{p}$, its $j$-th coordinate within the first and second space $\R^{\Z\setminus\{0\}}$ by $\boldsymbol L_{q_j}$ and $\boldsymbol L_{p_j}$, respectively, and the evaluation at $t\in \R$ of its penultimate and last coordinate by $\boldsymbol L_{N}(t)$ and $\boldsymbol L_F(t)$, respectively. Finally, let $\boldsymbol q, \boldsymbol p \in \R_{+}^{\Z\setminus\{ 0\}}$ be row vectors with $j$-th entry equal to $q_j$ and $p_j$, respectively, and let $\hat{\boldsymbol{q}}_n$ and $\hat{\boldsymbol{p}}_n$ be their coordinate-wise estimators.

\begin{thm}\label{ThmMain}
	Let $\mathbb L$ be a tight centred Gaussian random variable on $ \R^{\N} \times \ell^{\infty}(\R)^2$ and note that it is fully characterised by its finite dimensional distributions. For any $c_{(\lambda)}, c_{(\gamma)}, c_{(q)}, c_{(p)} \in \R$, $M_\theta\in \N$, $\boldsymbol C_{\theta}\in \R^{M_{\theta}}$, where $\theta\in\{ q,p,N, F\}$, and any $i_1, \ldots, i_{M_q}, j_1, \ldots, j_{M_p} \in \N,  s_{1}, \ldots, s_{M_N},$ $t_1, \ldots, t_{M_F}\in \R$, assume $\mathbb L$ satisfies that
	\begin{equation*}
		\big(c_{(\lambda)}, c_{(\gamma)}, c_{(q)}, c_{(p)} \big) \!\left( \!
		\begin{array}{c}
			\mathbb{L}_{\lambda} \\
			\mathbb{L}_{\gamma} \\
			\mathbb{L}_{q} \\
			\mathbb{L}_{p} 
		\end{array}
		\! \right) 
		+\boldsymbol C_q\!\left(\!
		\begin{array}{c}
			\mathbb{L}_{q_{i_1}} \\
			\vdots \\
			\mathbb{L}_{q_{i_{M_q}}} 
		\end{array}
		\!\right) 
		+\boldsymbol C_p \! \left(\!
		\begin{array}{c}
			\mathbb{L}_{p_{j_1}} \\
			\vdots \\
			\mathbb{L}_{p_{j_{M_p}}} 
		\end{array}
		\!\right) 
		+\boldsymbol C_N\!\left(\!
		\begin{array}{c}
			\mathbb{L}_{N}(s_1) \\
			\vdots \\
			\mathbb{L}_{N}(s_{M_{N}}) 
		\end{array}
		\!\right) 
		+ \boldsymbol C_F \!
		\left(\!
		\begin{array}{c}
			\mathbb{L}_{F}(t_1) \\
			\vdots \\
			\mathbb{L}_{F}(t_{M_F})
		\end{array}
		\!\right) 
	\end{equation*}
	is a one-dimensional normal random variable with mean zero and variance
	\begin{align}\label{eqVarL}
		\sigma_{\mathbb{L}}^2:=\frac{1}{{\Deltaup}^2}  \int_{\R} \big( f^{(\mathbb{L})} \ast {\cal F}^{-1} \left[ \varphi^{-1}(-\cdot) \right] (x) \big)^2 \, P(dx),
	\end{align}
	where, writing $f^{(q)}:=\sum_{j\in\Z\setminus\{0\}} f^{(q_j)}$ and $f^{(p)}:=\sum_{j\in\Z\setminus\{0\}} f^{(p_j)}=\lambda^{-1}(f^{(q)} - pf^{(\lambda)})$, we define 
	\begin{equation*}
		f^{(\mathbb{L})}:=\big(c_{(\lambda)}, c_{(q)}, c_{(p)}\big) \! \left( \!
		\begin{array}{c}
			f^{(\lambda)} \\
			f^{(q)} \\
			f^{(p)} 
		\end{array}
		\!  \right)
		+ \boldsymbol C_q\! \left(\!
		\begin{array}{c}
			f^{(q_{i_1})} \\
			\vdots \\
			f^{(q_{i_{M_q}})} 
		\end{array}
		\!\right) 
		+ \boldsymbol C_p \!\left(\!
		\begin{array}{c}
			f^{(p_{j_1})} \\
			\vdots \\
			f^{(p_{j_{M_p}})} 
		\end{array}
		\!\right) 
		+ \boldsymbol C_{N} \!\left(\!
		\begin{array}{c}
			f^{(N)}_{s_1} \\
			\vdots \\
			f^{(N)}_{s_{M_{N}}}
		\end{array}
		\!\right) 
		+ \boldsymbol C_{F} \!\left(\!
		\begin{array}{c}
			f^{(F)}_{t_1} \\
			\vdots \\
			f^{(F)}_{t_{M_F}}
		\end{array}
		\!\right).
	\end{equation*}
	
	\begin{enumerate} 
		\item Under the assumptions of Proposition \ref{Proplambdamu} we have that as $n\to \infty$
		\begin{equation*}
			\sqrt{n} \left(
			\hat \lambda_n -\lambda, \,
			h_n^{-1}(\hat \gamma_n-\gamma),\,
			\hat{\boldsymbol{q}}_n-\boldsymbol q ,\,
			\hat{\boldsymbol{p}}_n-\boldsymbol p ,\,
			\widehat{N}_n - N ,\,
			\widehat{F}_n - F 
			\right) \to^{\mathcal{L}^{\times}} \mathbb{L}_{-q,-p},
		\end{equation*}
		where $\mathbb{L}_{-q,-p}$ denotes the same random variable as $\mathbb{L}$ without its third and fourth coordinates. 
		
		\item Suppose $\nu$ satisfies $\int_{\R}\! |x|^{\beta} \nu(dx)\!<\!\infty$ for some $\beta\!>\!1$ and Assumption \ref{enum:th:ass1} for some $\alpha\!>\!8\beta/(\beta\!-\!1)$. 
		Assume $K\!$ is as in \eqref{conditionsK} and let $h_n \!\sim \!\exp(-n^{\vartheta_h}\!), \varepsilon_n \!\sim\!\exp(-n^{\vartheta_\varepsilon}\!), H_n\! \sim\!\exp(n^{\vartheta_H}\!)$ and $\widetilde H_n\! \sim\! n^{\vartheta_{\widetilde{H}}}\!$, where $1/(2\beta)\!\leq\! \vartheta_{\widetilde{H}} \!<\! 1/2$, $ 2/\alpha\!<\!2\vartheta_{\varepsilon} \!\leq\! \vartheta_{h} \!<\!(1\!-\!2\vartheta_{\widetilde{H}})/4$ and $ 0\!<\!\vartheta_{H}\!<\! \vartheta_{h}$. Then, as $n\!\to\! \infty$,
		\begin{equation*}
			\sqrt{n} \left(
			\hat \lambda_n -\lambda,\,
			h_n^{-1}(\hat \gamma_n-\gamma),\,
			\hat q_n - q,\,
			\hat p_n - p,\,
			\hat{\boldsymbol{q}}_n-\boldsymbol q ,\,
			\hat{\boldsymbol{p}}_n-\boldsymbol p ,\,
			\widehat{N}_n - N ,\,
			\widehat{F}_n - F 
			\right)
			\to^{\mathcal{L}^{\times}} \mathbb{L}.
		\end{equation*}
	\end{enumerate}
\end{thm}

A large number of inference procedures can be derived from this theorem either by applying bootstrap methods or by using `Studentisation' ideas. These include estimation, confidence regions, goodness of fit tests and two sample tests for single and several parameters at once. 
Since we assume that $\nu$ has at most two components, tests of the presence of each of them can be easily derived from the central limit theorems for $q$ or $p$. We refer the reader to Chapter 4 in \cite{C16} for more details of the construction, implementation and practical performance of all of these procedures. There, it is emphasised that, even though the distinguished logarithm featuring in the estimators can be implemented, the simpler and faster-to-compute principal branch of the complex logarithm can be used in most realistic cases when the value of $\lambda$ is moderate and $\gamma=0$.

We require the stronger tail assumption of part (b) to show convergence of the finite-dimensional distributions and hence regard it as purely technical: in part (a), we need to control a finite fixed number of projections of the quantity considered therein but, in part (b), we effectively have to control a growing number of projections due to the expressions for $\hat q_n$ and $\hat p_n$. This means $\widetilde H_n$ cannot grow exponentially but polynomially and, in turn, implies a finite moment condition is needed to control bias terms. The latter results in a parameter-free lower bound for $\vartheta_{H}$ and we remark that if such condition is imposed on part (a) or if Assumption \ref{enum:th:ass1b} is strengthened to H{\"o}lder regularity of order $\alpha>0$, the lower bounds on $\vartheta_{H}$ in part (a) or on $\vartheta_{\varepsilon}$ in both parts, respectively, simplify to $0$; if one assumes such regularity and tail decay are present in $F$, none of the bounds for the hyperparameters, including that for the bandwidth, depend on the unknown parameters $\alpha$ and $\beta$. This also applies to the bounds in the previous propositions.

\section{Proofs} \label{SecProofs}

Joint convergence as described in Theorem \ref{ThmMain} follows from convergence of each of the coordinates of the infinite vectors therein together with  joint convergence of the one-dimensional parameters and the finite-dimensional distributions of the infinite-dimensional parameters. Therefore we need to prove joint convergence of finitely and infinitely many one-dimensional distributions. This is the content of Section \ref{SecProofOneDimDns} and it includes several results that may be of independent interest. Section \ref{SecProofProps} is devoted to proving Propositions \ref{Proplambdamu} and \ref{Proppsqs}, which follow from the results of Section \ref{SecProofOneDimDns}, and in Sections \ref{SecProofNF} and \ref{SecProofJointConv} we prove Theorems \ref{ThmNF} and \ref{ThmMain}, respectively. Lastly, we include the proof of Lemma \ref{lemmaAsympSigma}. From now on we take $\varepsilon=1$ in Assumption \ref{enum:th:ass1a} for notational simplicity.

\subsection{Joint convergence of one-dimensional distributions} \label{SecProofOneDimDns}

The last and main result of this section requires to identify the asymptotic limit of some stochastic quantities and to control some non-stochastic quantities. We follow this order throughout and before introducing the generic central limit theorem dealing with the stochastic quantities, namely Theorem \ref{ThmGenericCLT}, we develop several auxiliary results. 

Note that the observations of the compound Poisson process appear in our estimators through the empirical characteristic function $\varphi_n$. Since it is always multiplied by $\FT K_{h_n}$, supported in $[-h_n^{-1}, h_n^{-1}]$, we have to control the uniform norm of $|\varphi_n(u) - \varphi(u)|$ when $u$ varies over these sets. Theorem 4.1 in \cite{NR09} gives sufficient conditions to control this quantity and derivatives of it. In particular, it guarantees that if $\int |x|^{\beta} \nu (x) dx< \infty$ for some $\beta>0$ then, under the notation at the end of Section \ref{SecDefNot},
\begin{equation}\label{PhintophiIdent}
	\sup_{\abs{u}\le h_n^{-1}}\abs{\varphi_n(u)-\varphi(u)}=O_{\Pr}\big(n^{-1/2}\log^{1/2+\delta} (h_n^{-1}) \big) \quad \mbox{ for any $\delta>0$.}
\end{equation}
The assumption for this particular case of the result can be slightly refined as follows.

\begin{thm} \label{ThmLogMoment}
	Suppose that $\xi_1, \xi_2, \ldots$ are independent and identically distributed real-valued random variables satisfying $\mathbb{E}[\log^\beta ( \max\{|\xi_1|,e\})]<\infty$ for some $\beta>1$ and let $w: \R \to [0,\infty)$ be a weight function with $w(u) \leq \log^{-1/2-\delta}(e+|u|)$ for some $\delta>0$. Then we have that
	\begin{equation}\label{PhinPhiIdent}
		\sup_{n \geq 1} \, \mathbb{E}  \Big[  \, \sup_{u\in \R} \left\{ \sqrt{n}\, |\varphi_n(u) - \varphi(u)| \,  w(u) \right\} \Big] <\infty,
	\end{equation}
	and, from Markov's inequality, \eqref{PhintophiIdent} holds.
\end{thm}

Taking $\xi_k\!=\!Z_k$, $k=1, \ldots, n$, the assumption of the theorem is that $\int_{\R} \log^\beta (\max\{|x|, e\})  P(dx)\!<\!\infty$ for some $\beta>1$. By Theorem 25.3 and Proposition 25.4 in \cite{S99} it is equivalent to assuming the same finite moment condition on $\nu$ and it is thus satisfied under Assumption \ref{enum:th:ass2}. 
We show this theorem by slightly refining one of the steps in the proof of Theorem 4.1 in \cite{NR09}. Therefore, we first give a brief overview of their ideas and take no credit for them. 

\begin{proof} 
	The strategy Neumann and Rei{\ss} \cite{NR09} adopt to prove the result is to control the expectation of the supremum over $u \in \R$ of the empirical process $\sqrt{n} \exp(i\cdot u) w(u) (P_n-P)$ by a maximal inequality from empirical process theory. Namely, they make use of Corollary 19.35 from \cite{AvdV98} and thus they have to control the $L^2$-bracketing number of a certain class of functions. Inspired by Yukich \cite{Y85} they bound this bracketing number by an expression depending on the quantity 
	\begin{equation*}
		M=M(\epsilon, k):= \inf\left\{ m\geq 1:  \mathbb{E}[\xi_1^{2k} \mathds{1}_{|\xi_1|>m}] \leq \epsilon^2 \right\},
	\end{equation*}
	where $k\in \N$ is fixed and we are only concerned with the case $k=0$. By the lemma after Theorem 2 in \cite{Y85}, the theorem follows if 
	\begin{equation}\label{MCondLogMoment}
		\int_0^1 \sqrt{\log M(\epsilon,k)} \, d\epsilon <\infty.
	\end{equation}
	For $k=0$ the expectation in the definition of $M$ simplifies to $P(|\xi_1|>m)$. Then, trivially $|y|\leq \max\{|y|,e\}$ for any $y\in \R$ and, by Markov's inequality, we have for any $\beta>1$
	\begin{equation*}
		P(|\xi_1|>m) \leq \frac{\mathbb{E}[ \log^\beta( \max\{|\xi_1|, e\} ) ]} {\log^{\beta}( m)}.
	\end{equation*}
	Once an upper bound of this form is derived, the result follows because
	\begin{equation*}
		M(\epsilon, 0) \leq  \inf\left\{ m\geq 1:  \frac{\mathbb{E}[ \log^\beta( \max\{|\xi_1|, e\} ) ]}{\log^{\beta}( m)} \leq \epsilon^{2} \right\} 
	\end{equation*}
	and hence
	\begin{equation*}
		\log M(\epsilon,0) \leq \epsilon^{-2/\beta} \, \mathbb{E}[ \log^\beta(\max\{|\xi_1|, e\} ) ]^{1/\beta}.
	\end{equation*}
	Therefore, \eqref{MCondLogMoment} is satisfied when assuming $\mathbb{E}[ \log^\beta(\max\{|\xi_1|, e\} ) ]<\infty$ for some $\beta>1$. 
\end{proof}

The quantity $\Log (\varphi_n/\varphi)$ plays a central role in the proofs as it carries all the stochasticity. To apply standard tools such as the central limit theorem for triangular arrays or tools from empirical process theory we need to linearise it, in the sense of writing it as a function applied to $P_n-P$. The next result guarantees its decomposition into a linear part and a `remainder term'. 

In line with the notation introduced at the end of Section \ref{SecDefNot}, we write $\xi_n=o_{\Pr}(r_n)$ to denote that $\xi_n r_n^{-1} $ vanishes as $n\to \infty$ in $\Pr$-probability and $\xi_n=^{\Pr} \xi_n'$ denotes equality in sets of $\Pr$-probability approaching $1$ as $n\to \infty$.

\begin{thm}\label{ThmLinearisation}
	Let $\xi_k$, $k=1, \ldots, n$, be independent and identically distributed real-valued random variables with characteristic function $\varphi$ such that $\inf_{u\in \R} |\varphi(u)|>0$. Let $\varphi_n$ be the empirical version of it and take $h_n\to 0$ such that  $\log (h_n^{-1})/n^{1/(1+2\delta)} \to 0$ for some $\delta>0$. If $\mathbb{E}[\log^{\beta} (\max\{|\xi_1|, e\})]< \infty$ for some $\beta>1$ the quantity $\Log (\varphi_n(u)/\varphi(u)) \1_{|u|\leq h_n^{-1}}$ is well-defined and finite in sets of $\Pr$-probability approaching one as $n\to \infty$. Furthermore, in these sets
	\begin{equation}\label{eqStochIdent}
		\Log \frac{\varphi_n(u)}{\varphi(u)} =^{\Pr} \left( \varphi_n(u) -\varphi(u) \right) \varphi^{-1}(u) +R_n(u), \qquad \quad u \in [-h_n^{-1},h_n^{-1}], 
	\end{equation}
	where the remainder term is given by
	\begin{equation*}
		R_n(u):=R(z)|_{z=\varphi_n(u)/\varphi(u)} \qquad \mbox{with} \qquad R(z)=\int_1^z \frac{s-z}{s^2}ds
	\end{equation*}
	and it is such that $\sup_{\abs{u}\le h_n^{-1}}\abs{R_n(u)} = O_{\Pr}(n^{-1}(\log h_n^{-1})^{1+2 \bar \delta})$ for any $\bar \delta\in (0,\delta]$.
\end{thm}

\begin{proof}
	From the construction of the distinguished logarithm (see Theorem 7.6.2 in \cite{C01}) we know that in the closed disk $|z-1|\leq 1/2$ it coincides with the principal branch of the complex logarithm. Thus, there it satisfies that
	\begin{equation*}
		\Log z = (z-1) + R(z),
	\end{equation*}
	where $R(z)$, the integral form of the remainder, is as in the statement of the theorem, so
	\begin{equation}\label{eq:boundremainder}
		|R(z)| \leq \sup_{s\in [1,z]} |s|^{-2} \int_{1}^z|z-s|ds \leq  \max\{1,|z|^{-2}\} |z-1|^2.
	\end{equation}
	Taking $z=\varphi_n(u)/\varphi(u)$, we see that the left hand side of \eqref{eqStochIdent} is well-defined in sets of $\Pr$-probability approaching one as $n\to \infty$ and \eqref{eqStochIdent} holds in them if we show that in such type of sets $|\varphi_n(u)/\varphi(u)-1|\leq 1/2 $ whenever $u \in [-h_n^{-1},h_n^{-1}]$. In view of the finite logarithmic moment assumption, the asymptotics of $h_n$ and Theorem \ref{ThmLogMoment}, we have that for any $\delta>0$
	\begin{equation}\label{PhintophioP1}
		\sup_{\abs{u}\le h_n^{-1}}\abs{\varphi_n(u)-\varphi(u)}=O_{\Pr}\big(n^{-1/2}\log^{1/2+\delta} (h_n^{-1}) \big)=o_{\Pr}(1).
	\end{equation}
	The conclusion then follows because, due to the strictly positive lower bounded on $|\varphi|$ in \eqref{eq:boundvarphi}, 
	\begin{equation}\label{PhintophiCond1}
		\left|\frac{\varphi_n(u)}{\varphi(u)}-1\right| \1_{|u|\leq h_n^{-1}} \leq \sup_{\abs{u}\le h_n^{-1}}\abs{\varphi_n(u)-\varphi(u)} \, \Big(\inf_{\abs{u}\le h_n^{-1}} \abs{\varphi(u)} \Big)^{-1} =o_{\Pr}(1).
	\end{equation}
	
	To show the second claim of the theorem we first note that evaluating \eqref{eq:boundremainder} at $z=\varphi_n(u)/\varphi(u)$, 
	\begin{equation}\label{IneqRemTerm}
		|R_n(u)| \leq \max\left\{1,\left|\frac{\varphi(u)}{\varphi_n(u)}\right|^{2}\right\} \left|\frac{\varphi_n(u)}{\varphi(u)}-1\right|^2.
	\end{equation}
	In view of \eqref{PhintophioP1}, the second term on the right hand side is of the required order. We claim that 
	\begin{equation}\label{eqinfphin}
		\inf_{\abs{u}\le h_n^{-1}}\abs{\varphi_n(u)}>\kappa \quad \mbox{ for every } \kappa\in \Big(0,\inf_{u \in \R}|\varphi(u)|\Big)
	\end{equation}
	in the same  $\Pr$-probability sets where \eqref{PhintophioP1} holds. Indeed, by \eqref{eq:boundvarphi} we have that for any $\epsilon>0$
	\begin{align*}
		\Pr \Big(\sup_{\abs{u}\le h_n^{-1}}\abs{\varphi_n(u)-\varphi(u)} \leq \epsilon \Big) & \leq \Pr \Big(\sup_{\abs{u}\le h_n^{-1}}\Big||\varphi_n(u)|-|\varphi(u)| \Big| \leq \epsilon \Big)\\
		& \leq \Pr \Big(|\varphi_n(u)|\geq |\varphi(u)|-\epsilon \, \text{ for all } u \in [-h_n^{-1}, h_n^{-1}] \Big)\\
		& \leq \Pr \Big(|\varphi_n(u)|\geq \inf_{u \in \R}|\varphi(u)|-\epsilon \,  \text{ for all } u \in [-h_n^{-1}, h_n^{-1}]\Big)\\
		& = \Pr \Big(\inf_{\abs{u}\le h_n^{-1}} |\varphi_n(u)| \geq \inf_{u \in \R}|\varphi(u)|-\epsilon \Big) \leq 1,
	\end{align*}
	and the conclusion follows by taking $\epsilon<\inf_{u \in \R}|\varphi(u)|$ and noting that the left hand side converges to $1$ as $n\to \infty$ by \eqref{PhintophioP1}.
\end{proof}

The following two lemmas are relatively long to prove or will be repeatedly used after and therefore we include them to improve the flow of other proofs.

\begin{lem}\label{LemmaFIphim1FTKh}
	Let $\varphi$ and $\FII{\varphi^{-1}(-\cdot)}$ be as in \eqref{eqLevyKhintchineCPP} and \eqref{eqFIPhim1nP} with $\gamma\in\R, \lambda>0$ and $\nu$ any finite measure on $\R$ satisfying $\nu(\R)=\lambda$. Let $K$ be a kernel function satisfying $K, \FT K  \in L^1$. Then for any $h>0$ fixed
	\begin{equation}\label{FIphim1FTKh=FIphim1astKh}
		\FII{\varphi^{-1}(-\cdot) \FT K_h }(x)=\FII{\varphi^{-1}(-\cdot) } \ast K_h \, (x) \quad \mbox{ for all } x\in \R
	\end{equation}
	and, consequently, if $K,g\in L^2(\R)$ we have that
	\begin{equation}\label{eqFIFTg...}
		\FII{\FT g \, \varphi^{-1}(-\cdot) \FT K_h}(x) = g \ast K_h \ast \FII{\varphi^{-1}(-\cdot)}(x) \quad \mbox{ for all } x\in \R.
	\end{equation}
\end{lem}

\begin{proof}
	By the expression for $\varphi$ given in \eqref{eqLevyKhintchineCPP} and under the notation introduced right after \eqref{eqFIPhim1nP},  $\varphi^{-1}(-u)= \exp \left({\Deltaup}(i u\gamma + \lambda - \FT \bar \nu(u))\right)$ for any $u \in \R$. Then, noting that $\FT \delta_{\gamma \Deltaup}(u)= \exp \left( i u \gamma \Deltaup \right)$ and writing the exponential of $- {\Deltaup} \FT \bar \nu (u)$ as an infinite series, we have for any $h>0$ fixed 
	\begin{align}\label{Identphim1infsum}
		\varphi^{-1}(-u)  = e^{\lambda \Deltaup} \FT \delta_{\gamma \Deltaup}(u) \sum_{k=0}^{\infty} (\FT \bar \nu(u) )^k\frac{(-{\Deltaup})^k}{k!}  = e^{\lambda \Deltaup} \FT \delta_{\gamma \Deltaup}(u)\sum_{k=0}^{\infty} \FTT{ \bar \nu^{\ast k}}(u) \frac{(-{\Deltaup})^k}{k!},
	\end{align} 
	where in the last equality we use that if $\mu_1$ and $\mu_2$ are finite measures on $\R$ then their convolution is a finite measure too and $\FT \mu_1 (u)\FT \mu_2 (u)= \FTT{\mu_1 \ast \mu_2}(u)$ for every $u\in \R$. These properties of convolution theory and Fourier analysis are repeatedly used in what follows and at the end of the proof we use that the same identity holds for $L^2(\R)$ functions too. To introduce the infinite sum into the Fourier transform we use that if $\mu_m$, $m=1,2, \ldots$, and $\mu$ are finite measures on $\R$ such that $\norm{\mu_m - \mu}_{TV} \to 0$ then $\FT \mu_m \to \FT \mu$ pointwise. Taking
	\begin{equation*}
		\mu_m := \sum_{k=0}^{m} \bar \nu^{\ast k} \frac{(-{\Deltaup})^k}{k!} \quad \mbox{ and } \quad  \mu := \sum_{k=0}^{\infty} \bar \nu^{\ast k} \frac{(-{\Deltaup})^k}{k!}
	\end{equation*}
	we see the condition is readily satisfied because
	\begin{equation*}
		\bnorm{ \sum_{k=m+1}^{\infty} \bar \nu^{\ast k} \frac{(-{\Deltaup})^k}{k!} }_{TV} \leq \sum_{k=m+1}^{\infty} \frac{(\lambda {\Deltaup})^k}{k!} \to 0 \qquad \mbox{as } m\to \infty
	\end{equation*}
	because it is the tail of a convergent series. Therefore
	\begin{align*}
		\varphi^{-1}(-u) \FT K_{h}(u)  = e^{\lambda \Deltaup} \FT \delta_{\gamma \Deltaup}(u) \FTT{ \sum_{k=0}^{\infty} \bar \nu^{\ast k} \frac{(-{\Deltaup})^k}{k!} }(u) \FT K_{h}(u)  = \FTT{ \FII{\varphi^{-1}(-\cdot)} \ast K_{h} }(u),
	\end{align*}
	and the first display in the lemma follows by taking the Fourier inverse on both sides, noting that $\varphi^{-1}(-\cdot) \FT K_{h}, \FII{\varphi^{-1}(-\cdot)} \ast K_{h} \in L^1(\R)$ and that the latter is continuous because $K,\FT K\in L^1(\R)$. The second display is then justified by the remark after \eqref{Identphim1infsum} and because \eqref{FIphim1FTKh=FIphim1astKh} is in $L^2(\R)$ if $K$ is by Minkowski's inequality for convolutions.
\end{proof}

\begin{lem}\label{Lemmalimsints}
	Let $\varphi$, $\FII{\varphi^{-1}(-\cdot)}$ and $P$ be given by \eqref{eqLevyKhintchineCPP} and \eqref{eqFIPhim1nP}, where $\gamma\in\R, \lambda>0$ and $\nu$ is a finite measure satisfying $\nu(\R)=\lambda$. Let $K$ be a symmetric kernel function satisfying $K \in L^1(\R) $ and $\supp \FT K \subseteq [-1,1]$, and let the bandwidth $h_n\to 0$. Let $g^{k)}_n \in L^2(\R)$, $k=1,2$, be functions satisfying  $\sup_{n} \sup_{x\in\R}|g^{k)}_n(x)| <\infty$ and such that $g^{k)}:= \lim_{n\to \infty} \, g^{k)}_n \ast K_{h_n}$ exists at every point. Then $g^{k)}$ is finite everywhere, 
	\begin{align}\label{LimCov}
		& \lim_{n\to \infty} \int_{\R} g^{1)}_n \ast K_{h_n} \ast \FII{\varphi^{-1}(-\cdot)} \, (x) \, \, g^{2)}_n \ast K_{h_n} \ast \FII{\varphi^{-1}(-\cdot)} \, (x) \, P(dx) \notag \\
		& = \int_{\R} g^{1)}  \ast \FII{\varphi^{-1}(-\cdot)} \, (x)\, \,  g^{2)} \ast  \FII{\varphi^{-1}(-\cdot)} \, (x) \, P(dx) 
	\end{align}
	exists and so does
	\begin{equation}\label{LimMean}
		\lim_{n\to \infty}  \int_{\R}  g^{k)}_n \ast K_{h_n} \ast \FII{\varphi^{-1}(-\cdot)} \, (x) \, P(dx) = \lim_{n\to \infty} g^{k)}_n \ast K_{h}\, (0)=:  g^{k)}(0), \quad k=1,2.
	\end{equation}
	Furthermore, if $\nu$ satisfies Assumption \ref{enum:th:ass1a} and, for $k=1,2$, $\tilde{g}^{k)}$ is a bounded function agreeing with $g^{k)}$ everywhere up to a zero Lebesgue-measure set disjoint from $\Z$ then
	\begin{align}\label{LimEqual}
		&\int_{\R} \tilde{g}^{1)}  \ast \FII{\varphi^{-1}(-\cdot)} \, (x) \, \, \tilde{g}^{2)} \ast  \FII{\varphi^{-1}(-\cdot)} \, (x) \, P(dx) \notag \\
		& = \int_{\R} g^{1)}  \ast \FII{\varphi^{-1}(-\cdot)} \, (x) \, \, g^{2)} \ast  \FII{\varphi^{-1}(-\cdot)} \, (x) \, P(dx).
	\end{align}
\end{lem}

\begin{proof}
	Let $k=1,2$ be fixed throughout. Note that $\norm{K_{h}}_{TV}=\norm{K}_{TV}<\infty$ for all $h>0$, $\sup_{n} \sup_{x\in\R}|g^{k)}_n(x)| <\infty$ by assumption, and $\norm{\FII{\varphi^{-1}(-\cdot)}}_{TV}<\infty$. Then, in view of Minkowski's inequality for integrals, $\sup_{n} \sup_{x\in \R} |g^{k)}_n\ast K_{h_n}(x)| <\infty$ and therefore $g^{k)}$ is finite everywhere and $\sup_{n} \sup_{x\in \R} |g^{k)}_n\ast K_{h_n}\ast \FII{\varphi^{-1}(-\cdot)} (x)|< \infty$. The existence of all the displays therefore follows and so does equality \eqref{LimCov} using dominated convergence.
	
	Note that $\norm{\FT K_h}_2\lesssim h^{-1/2} \norm{K}_1<\infty$, which in particular implies that $K\in L^2(\R)$ used below. Additionally, $\FT g \in L^2(\R)$ and $\varphi^{-1}\in L^{\infty}(\R)$ so, by H\"{o}lder's inequality, $\FT g \, \varphi^{-1}(-\cdot) \FT K_h \in L^1(\R)$ .Then, in view of \eqref{eqFIFTg...}, the symmetry of $K$ and the finiteness of $P$, the left hand side in \eqref{LimMean} equals the limit as $n\to \infty$ of
	\begin{align*}
		\frac{1}{2\pi} \int_{\R} \FT g^{k)}_n(-u) \, \FT K_{h_n}(u) \, \varphi^{-1}(u) \, \FT P (u) \, du= \int_{\R} g^{k)}_n(x) \, K_{h_n}(x) \, dx = g^{k)}_n \ast K_{h_n}\, (0),
	\end{align*}
	where the first equality follows by the symmetry of $K$ and Plancherel's formula due to $g, K \in L^2(\R)$. The limit of this quantity is $g^{k)}(0)$ by definition and the first part of the lemma follows.
	
	To show the last claim we note that because $\tilde{g}^{k)}$ is bounded we can apply Fubini's theorem to write the left hand side of \eqref{LimEqual} as
	\begin{equation*}
		\int_{\R^3} \tilde{g}^{1)}(x-y_1) \, \tilde{g}^{2)}(x-y_2) \, P(dx) \, \FII{ \varphi^{-1}(-\cdot)} (dy_1) \, \FII{ \varphi^{-1}(-\cdot)} (dy_2).
	\end{equation*}
	In view of Lemma 27.1 in \cite{S99} the product of the three measures gives a finite measure on $\R^3$ comprising a (possibly null) atomic component and a (possibly null) absolutely continuous component. Recall that the former may have atoms only at $ \Z^3$. In any of these point measures the functions $\tilde{g}^{1)}$ and $\tilde{g}^{2)}$ are evaluated at a value in $\Z$ and therefore there they coincide with $g^{1)}$ and $g^{2)}$ by assumption. In the rest of the values they agree up to a set of zero Lebesgue-measure but in these they are integrated with respect to an absolutely continuous measure and thus the conclusion follows.
\end{proof}

The following is the central result when dealing with the stochastic terms arising from proving joint convergence of one-dimensional parameters.

\begin{thm}\label{ThmGenericCLT}
	Let $\varphi$ and $\varphi_n$ be as in Section \ref{SecDefNot}, i.e. the characteristic function, and its empirical counterpart, of the ${\Deltaup}$-increments of a compound Poisson process with law $P$, empirical measure $P_n$, drift $\gamma \in \R$, intensity $\lambda>0$ and finite jump measure $\nu$ satisfying $\nu(\R)=\lambda$ and Assumption \ref{enum:th:ass2}. Let $K$ be a symmetric kernel function satisfying $K \in L^1(\R) $ and $\supp \FT K \subseteq [-1,1]$, and let the bandwidth $h_n\to 0$ be such that $\log(h_n^{-1})/n^{1/(1+2\delta)} \to 0$ for some $\delta>0$. Finally, let $g_n\in L^2  (\R)$ be a function satisfying
	\begin{equation}\label{assumpggenCLT}
		\int_{-h_n^{-1}}^{h_n^{-1}} \big| \FT g_n (u)\big| \, du =o\big(n^{1/2}(\log h_n^{-1})^{-(1+2\delta)}\big), \quad \sup_{n} \sup_{x\in\R}|g_n(x)| <\infty \quad  \mbox{and} \quad g:= \lim_{n\to \infty} \, g_n \ast K_{h_n} 
	\end{equation}
	exists at every point. Then we have that as $n\to \infty$
	\begin{align}
		\!\! \sqrt{n} \, \frac{1}{{\Deltaup}} \! \int_{\R} \! g_n(x) \,  \FII{\Log \frac{\varphi_n}{\varphi} \FT K_h}\!(x) dx  \, & =^{\Pr} \!\! \sqrt{n} \, \frac{1}{{\Deltaup}} \! \int_{\R} g_n \ast K_h \ast \FII{\varphi^{-1}(-\cdot)}\! (x)(P_n \! - \! P)(dx) \notag \\
		& \qquad + \sqrt{n} \, \frac{1}{2 \pi {\Deltaup}} \int_{\R} \FT g_n(-u) R_n(u) \FT K_h(u) \, du \label{eqThmGenCLT} \\
		&  \to^d \, N(0,\sigma_g^2), \label{eqThmGenCLT2}
	\end{align}
	where $R_n$ is as in Theorem \ref{ThmLinearisation} and $\sigma_g^2$ is finite and satisfies
	\begin{align*} 
		{\Deltaup}^2\sigma_g^2 = \! \int_{\R} \! \left( g \ast \FII{ \varphi^{-1}(-\cdot)} (x) - g (0) \right)^2 P(dx) = \! \int_{\R}  \left( g \ast \FII{ \varphi^{-1}(-\cdot)} \! (x) \right)^2 P(dx) - g (0)^2.
	\end{align*}
\end{thm}

\begin{proof}
	We first note that the conditions of Theorem \ref{ThmLinearisation} are readily satisfied when $\xi_k=Z_k$ because $\inf_{u \in \R} |\varphi(u)| \geq \exp(-2 \lambda {\Deltaup})>0$, $\supp(\FT K_h) \subseteq [-h^{-1}, h^{-1}]$ and we have assumed the same decay in $h_n$ as well as Assumption \ref{enum:th:ass2}. Then, using the compact support of $\FT K_h$, the quantity in the Fourier inverse transform of
	\begin{equation}\label{eqStochTerm}
		\frac{1}{{\Deltaup}} \int_{\R} g_n(x) \, \FII{\Log \frac{\varphi_n}{\varphi} \FT K_h}(x) dx
	\end{equation}
	is in $L^2(\R)$ in sets of $\Pr$-probability approaching $1$ as $n\to \infty$ and, in these, Plancherel's formula can be used, noting that $g_n \in L^2(\R)$, to write the last display as
	\begin{align}\label{eqStochTermDec}
		\frac{1}{2 \pi{\Deltaup}} \int_{\R} \FT g_n(-u)\Log \frac{\varphi_n(u)}{\varphi(u)} \FT K_h(u) du =^{\Pr}  & \frac{1}{2 \pi {\Deltaup}} \int_{\R} \FT g_n(-u)\left( \varphi_n(u) - \varphi(u)\right) \varphi^{-1}(u) \FT K_h(u) du \notag \\
		& + \frac{1}{2 \pi {\Deltaup}} \int_{\R} \FT g_n(-u) R_n(u) \FT K_h(u) du.
	\end{align}
	Due to $|\FT K_h| \leq \norm{K}_1 \1_{[-h^{-1}, h^{-1}]}$, the second summand is bounded by ${\Deltaup}^{-1}  \norm{K}_1 $ times
	\begin{align*}
		\! \sup_{|u|\leq h_n^{-1}}  \! |R_n(u)| \!\int_{-h_n^{-1}}^{h_n^{-1}} \! |\FT g_n(u)| \, du = O_{\Pr} \Big( n^{-1} (\log h_n^{-1})^{1+2\delta} \! \!  \int_{-h_n^{-1}}^{h_n^{-1}} \! |\FT g_n(u)| \, du \Big) = o_{\Pr} \big( n^{-1/2} \big),
	\end{align*}
	where the equalities are justified using Theorem \ref{ThmLinearisation} and the first condition in \eqref{assumpggenCLT}. Hence this summand is negligible in the asymptotic distribution of the left hand side of \eqref{eqThmGenCLT}. The first summand in \eqref{eqStochTermDec} is ${\Deltaup}^{-1}$ times
	\begin{align*}
		\frac{1}{2\pi }\int_{\R} \FT g_n(-u) \varphi^{-1}(u) \FT K_h(u) \FTT{P_n - P} (du)  = \int_{\R} \FII{\FT g_n \, \varphi^{-1}(-\cdot) \FT K_h} (x)\left(P_n - P \right)(dx),
	\end{align*}
	where the equality follows by the symmetry of $K$, because $P_n-P$ is a finite measure and due to $\FT g_n \, \varphi^{-1}(-\cdot) \FT K_h \in L^1(\R)$ by H\"{o}lder's inequality noting that $\FT g_n \in L^2(\R)$, $\varphi^{-1}\in L^{\infty}(\R)$ and $\norm{\FT K_h}_2\leq h^{-1/2} \norm{K}_1<\infty$. Note that \eqref{eqThmGenCLT} is then justified by Lemma \ref{LemmaFIphim1FTKh}. Therefore we point out that the stochastic quantity \eqref{eqStochTerm} is centred after linearisation and the Lindeberg-Feller central limit theorem applies to 
	\begin{equation}\label{eqCLTTerm}
		\frac{1}{{\Deltaup}} \int_{\R} g_n \ast K_h \ast \FII{\varphi^{-1}(-\cdot)} (x)\left(P_n - P \right)(dx)
	\end{equation}
	if 
	\begin{equation}\label{CondCLT}
		\sup_{x\in \R} \left| g_n \ast \FII{\varphi^{-1}(-\cdot) \FT K_{h_n}} (x) \right| =o(n^{1/2})
	\end{equation}
	and if
	\begin{equation}\label{LimVar}
		\lim_{n\to \infty} \int_{\R} \left( g_n \ast K_{h_n} \ast \FII{\varphi^{-1}(-\cdot)} (x) - g_n \ast K_{h_n}(0) \right)^2 P(dx)
	\end{equation}
	exists (see, for example, Proposition 2.27 in \cite{AvdV98}). Note that $\norm{K_{h}}_{TV}=\norm{K}_{TV}<\infty$ for all $h>0$ and, by assumption, $\sup_{n} \sup_{x\in\R}|g_n(x)| <\infty$ and $\norm{\FII{\varphi^{-1}(-\cdot)}}_{TV}<\infty$. Then, in view of Minkowski's inequality for integrals, $\sup_{n} \sup_{x\in \R} |g_n\ast K_{h_n}(x)| <\infty$ and therefore $g$ is finite everywhere and $\sup_{n} \sup_{x\in \R} |g_n\ast K_{h_n}\ast \FII{\varphi^{-1}(-\cdot)} (x)|< \infty$. Consequently, \eqref{CondCLT} is satisfied and \eqref{LimVar} follows by dominated convergence because of the third assumption in \eqref{assumpggenCLT}. This also justifies that \eqref{eqThmGenCLT2} holds for the first expression for $\sigma_g^2$. The second expression for this quantity follows by expanding the square in \eqref{LimVar} and using \eqref{LimMean}.
\end{proof}

The following auxiliary result is immediate to prove but it will be subsequently used several times. Therefore, and for the sake of clarity, we formulate it in the form of a lemma.

\begin{lem}\label{LemmaRemainderTerm}
	Let $a,b,U,V\in \R$ such that $U>1$, $V>0$ and $|a|,|b|< V$. Then, for $k=0,1$,
	\begin{equation*}
		\int_{-U}^U \big|\FTT{ \, (\cdot)^{\,k} \1_{[a,b]}} \big|(u) \, du \lesssim  U^{2k-1}V^{k+1} + (1-k)\log U.
	\end{equation*} 
\end{lem}

\begin{proof}
	Note that for all $u\in \R$
	\begin{equation*}
		\big|\FTT{ \, (\cdot)^{\,k} \1_{[a,b]}} \big| (u) \leq \norm{\, (\cdot)^{\,k} \1_{[a,b]}}_1 \leq V^{k} \big(|a| + |b|\big) \lesssim V^{k+1}.
	\end{equation*}
	Therefore, the conclusion for the case $k=1$ follows immediately. 
	
	Due to $U>1$, when $k=0$ we split the integral in the statement as the sum of that over $[-U^{-1}, U^{-1}]$ and $[-U,U]\setminus [-U^{-1}, U^{-1}]$. By the last display the former is bounded above by $U^{-1}V$, up to constants independent of $U$ and $V$. To bound the latter we note that
	\begin{align*}
		\int_{[-U,U]\setminus [-U^{-1}, U^{-1}]} \left| \frac{e^{iub}-e^{iua}}{iu}  \right| \, du & \lesssim \int_{U^{-1}}^U u^{-1} \, du \lesssim \log(U).
	\end{align*}
\end{proof}

The next lemma guarantees the functions $g_n$ we subsequently input into Theorem \ref{ThmGenericCLT} satisfy the necessary conditions therein. We denote pointwise equality between functions by $\equiv$.

\begin{lem}\label{Lemmaglambdagnu}
	Suppose $K$ satisfies \eqref{conditionsK} and $h_n, \varepsilon_n\to 0$, $H_n\to \infty$ are such that $h_n \sim\exp(-\vartheta_h)$ with $\vartheta_h<1/4$, $h_n \varepsilon_n^{-1}=o(1)$ and $h_n H_n=o(1)$, hence satisfying the assumptions of Theorem \ref{ThmGenericCLT} on $K$ and $h_n$. Then, for all $c\neq 0, j\in \Z\setminus\{0\}$ and $t\in \R$ the functions $f^{(\lambda)}_n$, $f^{(\gamma)}_n$, $f^{(q_j)}_n$ and $f^{(N)}_{t,n}$ defined in Section \ref{SecSettingEstimators} are bounded, belong to $L^2(\R)$ and satisfy \eqref{assumpggenCLT} for $g_n$ equal to each of them with the constant hidden in the notation small-$o$ of the first condition not depending on $t$ or $j$ when $g_n=f^{(N)}_{t,n}$ or $g_n=f^{(q_j)}_n$. Recalling the definitions of $f^{(\lambda)}$, $f^{(q_j)}$ and $f^{(N)}_{t}$ in Section \ref{SecCLT} and defining $f^{(\gamma)}:=0$, 
	we have that
	\begin{equation} \label{Limglambdagmu}
		f^{(\lambda)}:= \1_{\R\setminus \{0\}} \equiv \lim_{n\to \infty} \, f^{(\lambda)}_n \ast K_{h_n} , \quad  f^{(\gamma)}:= 0 \equiv \lim_{n\to \infty} \, f^{(\gamma)}_n \ast K_{h_n},
	\end{equation}
	
	\begin{equation}\label{Limgj}
		\quad f^{(q_j)}: =\1_{ \{   j\} } \equiv \lim_{n\to \infty} \, f^{(q_j)}_n \ast K_{h_n}, \, \,  j\in \Z,
	\end{equation}
	and
	\begin{equation}\label{Limftn}
		l_t:= \lim_{n\to \infty} \, f^{(N)}_{t,n} \ast K_{h_n} \equiv f^{(N)}_{t} - \frac{1}{2} \1_{\{t\}} \1_{\R\setminus   \Z  }(t).
	\end{equation}
	Furthermore, for any finite set $ \mathcal{T}\subset \R$, any $C_{(\lambda)}, C_{(\gamma)}, C_{(j)}, C_t  \in \R$, $j\in \Z\setminus\{0\}$ and $t\in  \mathcal{T}$, and any $\widetilde{H}_n\to \infty$ define the linear combination 
	\begin{equation}\label{defg}
		f_{t,n}:=C_{(\lambda)} f^{(\lambda)}_n + C_{(\gamma)} f^{(\gamma)}_n + \sum_{\substack{ |j| \leq \widetilde H_n \\ j \neq 0}} C_{(j)} f^{(q_j)}_n + \sum_{t\in  \mathcal{T}} C_t f^{(N)}_{t,n}.
	\end{equation}
	Then $f_{t,n}$ enjoys the above-mentioned properties the individual functions in it satisfy if either 
	\begin{enumerate}
		\item finitely many coefficients $C_{(j)}$ are not zero;
		\item or $\sup_{j\in \Z\setminus\{0\}} |C_{(j)}| <\infty$, $\widetilde{H}_n \sim n^{\vartheta_{\widetilde{H}}}$ with $\vartheta_{\widetilde{H}} <1/2$ and $\vartheta_{h} < (1-2\vartheta_{\widetilde{H}})/4$.
	\end{enumerate}
	In both cases
	\begin{equation*}
		\lim_{n\to \infty} \, f_{t,n} \ast K_{h_n} \equiv C_{(\lambda)} f^{(\lambda)} + C_{(\gamma)} f^{(\gamma)} + \sum_{j\in  \Z\setminus\{0\}} C_{(j)} f^{(q_j)} + \sum_{t\in  \mathcal{T}} C_t l_t.
	\end{equation*}
\end{lem}

\begin{proof}
	The conditions of Theorem \ref{ThmGenericCLT} on $K$ are trivially satisfied by those assumed here, and those on $h_n$ hold because
	\begin{equation*}
		\log(h_n^{-1})/n^{1/(1+2\delta)}\sim n^{\vartheta_h-1/(1+2\delta)}=o\big( n^{(-3+2\delta)/(4+8\delta)}\big)=o(1)
	\end{equation*}
	for any $\delta \in (0,3/2)$ due to $\vartheta_h<1/4$.
	
	To check \eqref{assumpggenCLT} for each of the individual functions note that for any $n$ fixed, any $j\in \Z\setminus\{0\}$ and $t\in \R$, the functions $f^{(\lambda)}_n$, $f^{(\gamma)}_n$, $f^{(q_j)}_n$ and $f^{(N)}_{t,n}$ have bounded supports and their absolute values are uniformly bounded by $\max\{1,|c|^{-1}\}<\infty$. Therefore each of them is bounded, square integrable and satisfies the second condition in \eqref{assumpggenCLT}. To check the first condition in \eqref{assumpggenCLT} we first note that $f^{(\lambda)}_n=\1_{[-H_n, H_n]} - \1_{(-\varepsilon_n, \varepsilon_n)}$. Then, by Lemma \ref{LemmaRemainderTerm} with $k=0$,
	\begin{equation*}
		\int_{-h_n^{-1}}^{h_n^{-1}} \big| \FT f^{(\lambda)}_n  (u) \big| \, du \lesssim h_n \, H_n + h_n \, \varepsilon_n + \log(h_n^{-1}) =O(\log h_n^{-1}),
	\end{equation*}
	where in the last equality we have used the asymptotics of $h_n$, $\varepsilon_n$ and $H_n$. Taking $k=1$ in Lemma \ref{LemmaRemainderTerm}, the same arguments guarantee that 
	\begin{equation*}
		\int_{-h_n^{-1}}^{h_n^{-1}} \big| \FT f^{(\gamma)}_n  (u) \big| \, du  \lesssim 1,
	\end{equation*}
	and taking $k=0$ we have that for any $j\in \Z \setminus\{0\}$ 
	\begin{equation*}
		\int_{-h_n^{-1}}^{h_n^{-1}} \big| \FT f^{(q_j)}_n  (u) \big| \, du = \int_{-h_n^{-1}}^{h_n^{-1}} \big| e^{iu   j} \FTT{\1_{[-\varepsilon_n, \varepsilon_n]}} (u)\big| \, du  \lesssim h_n \, \varepsilon_n + \log(h_n^{-1}) =O(\log h_n^{-1}).
	\end{equation*}
	To justify that $f^{(N)}_{t,n}$ enjoys the same asymptotic property, note that it can be written as $\1_{[a_1,b_1]} +\1_{[a_2,b_2]}$ for some $a_1, a_2, b_1, b_2\in \R$ satisfying that their absolute values are bounded above by $H_n+\varepsilon_n$. Then, using Lemma \ref{LemmaRemainderTerm} with $k=0$ and the asymptotics of $h_n$, $\varepsilon_n$ and $H_n$,
	\begin{equation*}
		\int_{-h_n^{-1}}^{h_n^{-1}} \big| \FT f^{(N)}_{t,n}  (u) \big| \, du \lesssim h_n \, (H_n+\varepsilon_n) + \log(h_n^{-1}) =O(\log h_n^{-1}) \quad \mbox{ for any $t\in \R$.}
	\end{equation*}
	Then, these functions satisfy the first condition in \eqref{assumpggenCLT} because $h_n \sim \exp(-n^{\vartheta_h})$, $\vartheta_h<1/4$, so
	\begin{equation*}
		(\log h_n^{-1})^{2(1+\delta)}n^{-1/2} \sim n^{2(1+\delta)(\vartheta_h-\frac{1}{4(1+\delta)})}=o(1)
	\end{equation*}
	for some $\delta \!\in\!(0, 1/(4\vartheta_h)-1)$, for which we also have that $\log h_n^{-1}/n^{1/(1+2\delta)} \to 0$ since it is $o(n^{-\vartheta_h(1+2\theta_h)/(1-2\theta_h)})$. The lack of dependence on $j$ or $t$ of the constants hidden in the notation $\lesssim$ is clear by the arguments we have employed. 
	
	We now compute the limits in \eqref{Limglambdagmu} in reverse order. Notice that for any $x\in \R$
	\begin{equation*}
		h_n^{-1} \cdot \mathds{1}_{(-h_n,h_n)} \ast K_{h_n} \, (x) = \frac{x}{h_n}\int_{x/h_n-1}^{x/h_n+1}  K(y)\, dy - \int_{x/h_n-1}^{x/h_n+1} y\,  K(y)\, dy.
	\end{equation*}
	When $x=0$ this equals $- \int_{[-1,1]} \cdot \, K$, which is zero by the symmetry of $K$, and otherwise it is of order $O(h_n^{\eta -1})$ by the decay of $|K|$. Hence the second limit in \eqref{Limglambdagmu} is justified for any $c\neq 0$ because $\eta>2$ and $h_n\to 0$. To show the first limit we start by noting that for any function $g_n$ such that $\sup_n\sup_{x\in\R} |g_n(x)|<\infty$ we have that as $n\to \infty$
	\begin{equation}\label{LimglambdaHn}
		\big| \big( g_n - g_n \1_{[-H_n, H_n]}\big)\ast K_{h_n} \, (x) \big|  \leq \sup_n\sup_{x\in\R} |g_n(x)| \int_{\substack{\\ \\ \big[\frac{x-H_n}{h_n}, \frac{x+H_n}{h_n}\big]^{C}}} \qsub \big| K(y)\big| \,  dy \to 0 \quad \mbox{for all } x\in \R,
	\end{equation}
	because $K\in L^1(\R)$ and $h_n^{-1}, H_n\to \infty$. Therefore the limit of  $f^{(\lambda)}_n \ast K_{h_n}$ as $n\to \infty$ equals that of $\tilde f^{(\lambda)}_n\ast K_{h_n}$ at every point, where $\tilde f^{(\lambda)}_n:= \1_{\R\setminus(-\varepsilon_n, \varepsilon_n)}$. Additionally, for any $x\in \R$,
	\begin{equation*}
		\1_{\R\setminus(-\varepsilon_n, \varepsilon_n)} \ast K_{h_n}(x) = 1- \1_{(-\varepsilon_n, \varepsilon_n)} \ast K_{h_n}(x)
	\end{equation*}
	because $\int_{\R} K =1$, so we only need to compute the limit of the convolution on the right hand side. This coincides with \eqref{Limgj} when $j=0$, so we now compute \eqref{Limgj} for any $j\in \Z$. Note that
	\begin{align*}
		\mathds{1}_{(  j-\varepsilon_n,  j+\varepsilon_n)} \ast K_{h_n}\, (x)= \int_{(x-  j-\varepsilon_n)/h_n}^{(x-  j+\varepsilon_n)/h_n}  K(y) \, dy \quad \mbox{ for any $x\in \R$.}
	\end{align*}
	By the decay of $|K|$, if $x\neq   j$ this is of order $O(\varepsilon_n h_n^{\eta -1})$ and if $x=  j$ then $1$ minus this display equals $\int_{[-\varepsilon_n/h_n,\varepsilon_n/h_n]^c} K$ due to $\int K=1$. Using that $h_n,\varepsilon_n\to 0$, $\eta>2$, $h_n \varepsilon_n^{-1}=o(1)$ and $K\in L^1(\R)$ we thus conclude that
	\begin{align}\label{glambda}
		\lim_{n\to \infty} \, \mathds{1}_{(  j-\varepsilon_n,  j+\varepsilon_n)} \ast K_{h_n} \equiv \1_{\{   j\}},
	\end{align}
	and consequently the first limit in \eqref{Limglambdagmu} and that in \eqref{Limgj} follow.
	
	To show \eqref{Limftn} note that, in view of \eqref{LimglambdaHn} because $\sup_n\sup_{x\in\R} |f^{(N)}_{t,n}(x)|<\infty$, the limit of $f^{(N)}_{t,n}\ast K_{h_n}$ as $n\to \infty$ equals that of $\tilde{f}^{(N)}_{t,n}\ast K_{h_n} $ at every point, where
	\begin{equation}\label{tildeftnandutnbias}
		\!\!\!\!\!\tilde{f}^{(N)}_{t,n}\!:=\! \1_{\R \setminus (-\varepsilon_n, \varepsilon_n)} \1_{(-\infty, u]} \quad \!\mbox{and} \!\quad u\!=\!u(t,n)\!:=\! \left\{  \begin{array}{lll}
			t & \mbox{if} \!\!\!\quad |t-j|>\varepsilon_n & \mbox{for all } j\!\in\! \Z, \\
			j-\varepsilon_n & \mbox{if} \!\!\!\quad  j-\varepsilon_n \leq t < j & \mbox{for some } j\!\in\! \Z, \\
			j+\varepsilon_n & \mbox{if} \!\!\!\quad j \leq t \leq j+\varepsilon_n & \mbox{for some } j\!\in\! \Z\!.
		\end{array}
		\right.\!
	\end{equation}
	Therefore, when $t= j$ for some $j\in \Z$ we have
	\begin{equation}\label{tildeftnj}
		\tilde{f}^{(N)}_{t,n} = \1_{(-\infty,  j-\varepsilon_n]} + \1_{( j-\varepsilon_n,  j+\varepsilon_n)}  - \1_{(-\varepsilon_n, \varepsilon_n)}\1_{[0,\infty)}(j) ,
	\end{equation}
	almost everywhere and hence the following arguments are not affected. Due to $\varepsilon_n \to 0$ as $n\to \infty$, for any other $t$ fixed we can assume 
	\begin{equation}\label{tildeftnnoj}
		\tilde{f}_{t,n}^{(N)} = \1_{(-\infty, t]}-\1_{(-\varepsilon_n, \varepsilon_n)} \1_{(0,\infty)}(t).
	\end{equation}
	Thus, and in view of \eqref{glambda}, we only need to compute the limit of 
	\begin{equation*}
		\1_{(-\infty, y]} \ast K_{h_n} (x) = \int_{(x-y)/h_n}^{\infty} K(z)\, dz
	\end{equation*}
	when $y= j-\varepsilon_n$ for some $j\in \Z$ and when $y=t$. For the former case we have $(x-y)/h_n= (x- j)/h_n +\varepsilon_n/h_n$. Hence the limit of the last display is $\int_{\R} K=1$ if $x< j$ and zero otherwise, which can be written as $\1_{(-\infty,  j)}$, and the limit of \eqref{tildeftnj} convolved with $K_{h_n}$ is $\1_{(-\infty, t]}\1_{\R\setminus\{0\}}$. When $y=t$ the same arguments apply to $x\neq t$ but when $x=t$ we obtain $\int_{\R^+}K=1/2$ by the symmetry of $K$. This gives the limiting function $\1_{(-\infty, t)}+ \frac{1}{2}\1_{\{t\}}$ and therefore the limit of \eqref{tildeftnnoj} convolved with $K_{h_n}$ is $\1_{(-\infty, t]}\1_{\R\setminus\{0\}} -  \frac{1}{2}  \1_{\{t\}}$, thus justifying \eqref{Limftn}. 
	
	The statement regarding the linear combination follows immediately for case (a). To show it for case (b) note that because $\varepsilon_n\to 0$ we can assume $\varepsilon_n<1/2$ and therefore for any $n$ fixed
	\begin{equation*}
		\sup_n\sup_{x\in \R} |f_{t,n}(x)| \leq \mbox{card}( \mathcal{T}) + 1 + 1\!<\!\infty \quad \mbox{and} \quad \norm{f_{t,n}}_2 \leq \mbox{card}( \mathcal{T}) (H_n+\varepsilon_n)^{1/2} + \widetilde{H}_n  \varepsilon_n^{1/2} + h_n^{1/2}\!< \!\infty.
	\end{equation*}
	Notice that assumptions $\vartheta_{\widetilde{H}}<1/2$ and $\vartheta_h<(1-2\vartheta_{\widetilde{H}})/4$ imply
	\begin{equation*}
		\widetilde{H}_n \, n^{-1/2} (\log h_n^{-1})^{2(1+\delta)} \sim n^{\vartheta_{\widetilde{H}}-1/2 + 2 (1+\delta)\vartheta_h } = o(1)
	\end{equation*}
	for any $\delta\in \big(0, (1-2\vartheta_{\widetilde{H}})/(4\vartheta_h) -1\big)$. This interval has non-empty intersection with that under which the first display of the proof holds and consequently for some of those $\delta>0$
	\begin{equation*}
		\int_{-h_n^{-1}}^{h_n^{-1}} |\FT f_{t,n}(u)| \, du = O(\widetilde{H}_n \log(h_n^{-1})) = o\big(n^{1/2} (\log h_n^{-1})^{-(1+2\delta)}\big)
	\end{equation*}
	as required. To compute the limit of $f_{t,n}\ast K_h$ we only need to work with the second sum in $f_{t,n}$ and we recall that for any $j\in \Z \setminus\{0\}$ 
	\begin{equation*}
		f^{(q_j)}_n \ast K_{h_n}(x) = \left\{  \begin{array}{ll}
			\int_{-\varepsilon_n/h_n}^{\varepsilon_n/h_n} K & \mbox{if } x=  j, \\
			& \\
			O(\varepsilon_n h_n^{\eta-1}) & \mbox{otherwise,} 
		\end{array}
		\right.
	\end{equation*}
	where $\eta>2$. Moreover, note that the assumption on the asymptotics of $\widetilde{H}_n$ implies $\widetilde{H}_n=o(n^{1/2})$ and therefore $\widetilde{H}_n h_n^{\epsilon}=o(1)$ for any $\epsilon>0$. Due to $\sup_{j\in \Z\setminus\{0\}} |C_{(j)}|<\infty$ and because $1-\int_{-\varepsilon_n/h_n}^{\varepsilon_n/h_n} K=o(1)$ the result follows in view of
	\begin{equation*}
		\! \sum_{|j|\leq \widetilde{H}_n } \!\! \! \! C_{(j)} \,\! f^{(q_j)}_n \ast K_{h_n}(x) \!=\! \left\{  \begin{array}{ll}
			\!\! C_{(j)}\! +\! o(1)\!+\!O\big( (\widetilde{H}_n-1) h_n^{\eta-2} \varepsilon_n h_n\big)\!=\!C_{(j)} \! +\! o(1) & \mbox{if } x=  j \mbox{ for} \\
			& \mbox{some $j\in \Z\setminus\{0\}$},\\
			\! O\big( \widetilde{H}_n h_n^{\eta-2} \varepsilon_n h_n\big)\!=\!o(1) & \mbox{otherwise.} 
		\end{array}
		\right.
	\end{equation*}
\end{proof}

All the results so far are developed to deal with stochastic quantities but, because of the use of a kernel in the estimators, we also need to control some non-stochastic quantities or bias terms. This is the content of the following lemma.

\begin{lem}\label{Lemmabiasterms}
	
	Adopt the setting and notation of Section \ref{SecDefNot} and assume $\nu$ satisfies Assumption \ref{enum:th:ass1a} for some $\alpha>0$. Let $K$ satisfy \eqref{conditionsK} and take $h_n\sim \exp(n^{-\vartheta_h})$, $\varepsilon_n\sim \exp(n^{-\vartheta_\varepsilon})$ and $H_n\sim \exp(n^{\vartheta_H})$ such that $1/\alpha<2 \vartheta_\varepsilon \leq\vartheta_h<\infty$ and $0<\vartheta_H <\vartheta_h $. Suppose that either
	\begin{enumerate}
		\item Assumption \ref{enum:th:ass2} holds for some $\beta>0$ and $ \vartheta_H \geq 1/(2\beta)$;
		\item or $\int_{\R} |x|^\beta \nu(dx)<\infty$ for some $\beta>0$.
	\end{enumerate}
	Define
	\begin{align*}
		B^{(\lambda)}_n := \frac{1}{{\Deltaup}} \int_{\R} f^{(\lambda)}_n (x) \FII{\Log \varphi \FT K_{h_n}} (x) \, dx - \lambda,
	\end{align*}
	\begin{equation*}
		\quad B^{(\gamma)}_n :=\frac{1}{ {\Deltaup}} \int_{\R} f^{(\gamma)}_n (x) \FII{\Log \varphi \FT K_{h_n}} (x) \, dx -h_n^{-1}\gamma, 
	\end{equation*}
	
	\begin{equation*}
		B^{(q_j)}_n := \frac{1}{{\Deltaup}} \int_{\R} f^{(q_j)}_n(x) \FII{\Log \varphi \FT K_{h_n}} (x) \, dx - q_j, \qquad j\in \Z\setminus\{0\},
	\end{equation*}
	and
	\begin{equation*}
		B^{(N)}_{t,n} := \frac{1}{{\Deltaup}} \int_{\R} f^{(N)}_{t,n}(x) \FII{\Log \varphi \FT K_{h_n}} (x) \, dx - N(t), \qquad t\in \R. 
	\end{equation*}
	Then, all these individual quantities are of order $o\big(n^{-1/2}\big)$ and so are $\sup_{j\in \Z\setminus\{0\}} |B^{(q_j)}_n|$ and $\sup_{t\in \R}|B^{(N)}_{t,n}|$. Furthermore, for any finite set $\mathcal{T}\subset \R$, any $C_{(\lambda)}, C_{(\gamma)}, C_{(j)}, C_t\in \R$, $j\in \Z\setminus\{0\}$ and $t\in \mathcal{T}$, and any $\widetilde{H}_n\to \infty$ define the linear combination
	
	\begin{equation}\label{eqbias}
		B_n:= C_{(\lambda)} B^{(\lambda) }_n + \gamma B^{(\gamma) }_n+\sum_{ j \in \Z\setminus\{0\}} C_{(j)} \big( (B^{(q_j)}_n+q_j) \1_{|j| \leq \widetilde H_n }  - q_j\big) +  \sum_{t\in \mathcal{T}} C_t B^{(N)}_{t,n}. 
	\end{equation}
	Then, $ B_n = o\big(n^{-1/2}\big)$ if either 
	\begin{enumerate}[label*=(\alph**)]
		\item condition (a) above is satisfied and finitely many coefficients $C_{(j)}$ are not zero;
		\item or condition (b) above is satisfied, $\sup_{j\in \Z\setminus\{0\}} |C_{(j)}| <\infty$,  $\widetilde{H}_n\sim n^{\vartheta_{\widetilde{H}}}$ for some $\vartheta_{\widetilde{H}}\geq 1/(2\beta)$ and, moreover, $\vartheta_{\varepsilon}>(1+2\vartheta_{\widetilde{H}})/(2\alpha)$.
	\end{enumerate}
	
\end{lem}

\begin{proof}
	For any $g=g_n$ bounded and by the same arguments that justify the first equality in \eqref{eqHeuristicsEstimators}, 
	\begin{align}\label{eqbiasmain}
		\frac{1}{{\Deltaup}} \int_{\R} g (x) \FII{\Log \varphi \FT K_{h_n}} (x) \, dx =  &  - \gamma \int_{\R} g (x) ( K_{h_n })'(x) -\lambda \int_{\R} g (x) K_{h_n }(x) \, dx \\
		& + \sum_{\mathclap{j\in  \Z\setminus\{0\}}} \,  q_j  \int_{\R} g (x) K_{h_n }(x -  j) \, dx + \int_{\R} g (x) \,  \nu_{ac} \ast K_{h_n}\, (x) \, dx, \notag
	\end{align}
	where we have swapped the infinite sum with the integral because for any $m\in \N$ and $x\in \R$
	\begin{equation*}
		\Bigg|\sum_{\substack{ |j| \leq m \\ j \neq 0}} q_j  g (x) K_{h_n }(x -  j)\Bigg| \leq \sup_{y\in \R} |g(y)| \sum_{j\in \Z\setminus \{0\}} q_j  \left|K_{h_n }\right|(x -  j)
	\end{equation*}
	and that the right hand side is integrable because $\norm{K_h(\cdot - y)}_1=\norm{K}_1$ for all $y\in \R$ and $\sum_{j\in \Z\setminus\{0\}} q_j \leq \lambda<\infty$. 
	
	We start by computing $B^{(\gamma) }_n$. Taking $g=g_n=h_n f^{(\gamma)}_n$ in \eqref{eqbiasmain}, using integration by parts in the first summand and, in view of $\nu_{ac}, K\in L^1(\R)$, Fubini's theorem in the last one, we have that
	\begin{align*}
		c  B^{(\gamma) }_n\!= & \, \gamma \, h_n^{-1} \bigg( \int_{ -1}^{1} K (x) \, dx -\big( K(-1) \!+\! K(1) \big) \!-\! c \bigg)   - \lambda \!\int_{- 1}^{1}\!x \, K (x) \, dx  \\
		& +  \sum_{\mathclap{j\in  \Z\setminus\{0\}}} \, q_j  \int_{\R} \! \frac{h_n x +   j }{h_n}\1_{(-h_n, h_n)}(h_n x +   j )  K(x) \, dx + \!\int_{\R} \!K(y)  \!\int_{\substack{  	\\ \\  -h_n y+h_n}}^{\substack{ - h_n y-h_n \\ \\ }} \!\!\!\!\! \qsub \qsub h_n^{-1} (x+h_n y)  \nu_{ac} (x) \, dx \, dy.
	\end{align*}
	By the symmetry and integrability of $K$, the decay of $|K|$ in \eqref{conditionsK}, the fact that $c:= 2(\int_{[0,1]}K - K(1) )$ and using that Assumption \ref{enum:th:ass1b} is satisfied for some $\alpha>0$ we conclude that
	\begin{equation*}
		|B^{(\gamma) }_n| \lesssim h_n^{\eta} \sum_{j\in  \Z\setminus\{0\}}  \,  q_j  \, j^{-\eta}  + (\log(h_n^{-1}))^{-\alpha}=O\big((\log(h_n^{-1}))^{-\alpha}\big) =o\big( n^{-1/2}\big)
	\end{equation*}
	for some $\eta\!>\!2$, where the last equality follows because $h_n\!\sim\! \exp(n^{-\vartheta_h})$ for some $\vartheta_h\!>\!1/\alpha\!>\!1/(2\alpha)$. 
	
	To compute $B^{(\lambda) }_n$ we write $f^{(\lambda)}_n=1-\1_{[-H_n, H_n]^C}-\1_{(-\varepsilon_n,\varepsilon_n )} $ and analyse \eqref{eqbiasmain} when $g$ equals each of these quantities. When $g=1$ we immediately see it equals $0$ because $K$ is symmetric, $\int_{\R}K=1$ and $\int_{\R} \nu=\lambda$. To analyse the second summand we first consider \eqref{eqbiasmain} for the more general function $ g=g_n= \tilde g_n \1_{[-H_n, H_n]^C}$, where $\tilde g_n$ is a function satisfying $\sup_n \sup_{x\in\R} |\tilde g_n(x)|<\infty$. This generalisation is used later on. In this case the right hand side of \eqref{eqbiasmain} equals
	\begin{align}\label{eqbiastailsvanish}
		-\gamma \, \int_{\substack{ \quad \\ \quad \\ [-H_n, H_n]^C}} \qsub \qsub  \tilde g_n(x) (K_{h_n})'(x) \, dx  - \lambda \int_{\substack{ \quad \\ \quad \\ [-H_n, H_n]^C}} \qsub \qsub \tilde g_n (x) K_{h_n}(x) \, dx + \int_{\R} K(y) \int_{\substack{ \quad \\ \quad \\ [-H_n-h_n y, H_n-h_n y]^C}} \qqsub  \tilde g_n(x+h_n y) \, \nu(dx)\,  dy. 
	\end{align}
	The value of the first integral depends on the function $\tilde g_n$. In this calculation $\tilde g_n=1$, so it equals $K_{h_n}(-H_n)-K_{h_n}(H_n)$, which is zero by the symmetry of $K$. By the decay of $|K|$ in \eqref{conditionsK}, the absolute value of the second integral is bounded above by
	\begin{equation*}
		\sup_{x\in\R} |\tilde g_n(x)|   \int_{\substack{ \quad \\ H_n/h_n}}^{\substack{ \infty \\}} \!\!\!\!  |K|(x) \, dx   = O\big( (h_n/H_n)^{\eta-1}\big) =o\big( n^{-1/2}\big)
	\end{equation*}
	for some $\eta>2$, where the last equality is justified by the exponential decay of $h_n$ and $H_n\to \infty$. The absolute value of the last term in \eqref{eqbiastailsvanish} is bounded, up to constants independent of $n$, by
	\begin{equation*}
		\sup_{x\in\R} |\tilde g_n(x)| \Big( \int_0^{h_n^{-1}} |K| (x) \, dx \int_{H_n}^\infty \nu(dx) + \lambda \int_{h_n^{-1}}^\infty |K| (x) \, dx \Big).
	\end{equation*}
	If condition (a) of the lemma is satisfied $\int_{H_n}^\infty \nu(dx)\leq (\log H_n)^{-\beta} \int_{H_n}^\infty (\log x)^{\beta} \nu(dx)=o\big( n^{-1/2}\big)$. If instead condition (b) holds $\int_{H_n}^\infty \nu(dx) \leq (H_n)^{-\beta} \int_{H_n}^\infty  x^{\beta} \nu(dx)=o\big( n^{-1/2}\big)$. Therefore, and using the decay of $|K|$ and $h_n$, the last display is $o\big( n^{-1/2}\big)$. Consequently we conclude that the second summand in the decomposition of $f^{(\lambda)}_n$ is $o\big(n^{-1/2}\big)$. To analyse the third one we need to compute \eqref{eqbiasmain} when $g_n=f^{(q_0)}_n$. Hence, we first compute $B^{(q_j)}_n$ for any $j\in \Z$ and, only during this calculation and with some abuse of notation, we take $q_0=-\lambda$ in its expression for notational purposes. In view of \eqref{eqbiasmain} we have that 
	\begin{align*}
		B^{(q_j)}_n= & \,\gamma \, h_n^{-1} \bigg(  K\Big(\frac{ j -\varepsilon_n}{h_n}\Big) -  K\Big(\frac{ j +\varepsilon_n}{h_n}\Big) \bigg)  + q_j \int_{\substack{ \quad \\  ( - \varepsilon_n/h_n, \varepsilon_n/h_n)^C}} \qsub \qsub   K (x) \, dx  \\
		& + \sum_{l\in  \Z\setminus\{j\}}  \,  q_l  \int_{\substack{ \quad \\  ( (j-l) - \varepsilon_n)/h_n}}^{\substack{ ( (j-l) + \varepsilon_n)/h_n \\ \quad}}  \qsub  K (x) \, dx  + \int_{\R} K(y)  \int_{\substack{ \quad \\  j -h_n y - \varepsilon_n}}^{\substack{  j - h_n y + \varepsilon_n\\ \quad}}  \qsub \nu_{ac} (x) \, dx \, dy.
	\end{align*}
	Then, by the symmetry of $K$, the decay of $|K|$, the fact that $\int_{\R} K=1$, Assumption \ref{enum:th:ass1b} and the integrability of $K$, the last display can be bounded, up to constants independent of $j$ and $n$, by
	\begin{align}\label{eq|Bj|bias}
		|B^{(q_j)}_n| & \lesssim   + \1_{\Z\setminus\{0\}}(j) \, h_n^{\eta-1}+ |q_j| \, (h_n/\varepsilon_n)^{\eta-1} + h_n^{\eta-1} \sum_{l\in  \Z\setminus\{0\}}  \,  q_{j-l}  \, l^{1-\eta} + (\log(\varepsilon_n^{-1}))^{-\alpha} \notag \\
		& \lesssim h_n^{\eta-1} + \lambda \, (h_n/\varepsilon_n)^{\eta-1} + \lambda \, h_n^{\eta-1} \sum_{l\in  \Z\setminus\{0\}}  l^{1-\eta} +(\log(\varepsilon_n^{-1}))^{-\alpha} =o\big(n^{-1/2}\big)
	\end{align}
	for some $\eta>2$ and $\alpha>0$, where the last equality follows because $h_n\sim\exp(-n^{\vartheta_h})$, $\varepsilon_n\sim \exp(-n^{\vartheta_\varepsilon})$ with $\vartheta_h \geq 2 \vartheta_\varepsilon>\vartheta_\varepsilon$ and $\vartheta_\varepsilon>1/(2\alpha)>0$. This ends showing that $B^{(\lambda) }_n=o\big(n^{-1/2}\big)$ and, in addition, it shows that $B^{(q_j) }_n$, $j\in \Z\setminus\{0\}$, and its supremum are of the same order.
	
	To analyse $B^{(N)}_{t,n}$ we work directly with its supremum over $t\in\R$. We first remove the truncation in the tails of $f^{(N)}_{t,n}$, just as we have done for $B^{(\lambda) }_n$. Take $\tilde g_n\!=\!\tilde g_{t,n}\!=\!\tilde{f}_{t,n}^{(N)}$ in \eqref{eqbiastailsvanish}, where $\tilde{f}_{t,n}^{(N)}$ is defined in \eqref{tildeftnandutnbias}. Note that the upper bounds for the absolute values of the last two summands of \eqref{eqbiastailsvanish} only depend on $t$ through $\sup_{x\in\R} |\tilde g_{t,n}(x)|$. Furthermore $ \sup_n \sup_{t\in \R} \sup_{x\in \R}|\tilde{g}_{t,n}(x)|<\infty$, so the supremum of these two terms are of order $o\big(n^{-1/2}\big)$ by the same arguments used therein. Therefore, if we show that the supremum over $t\in \R$ of the absolute value of the first summand is $o\big(n^{-1/2}\big)$, we only need to analyse the supremum of \eqref{eqbiasmain} for $g=g_{t,n}=\tilde{f}_{t,n}^{(N)}$. But the former follows easily noting that the absolute value of the first summand in \eqref{eqbiastailsvanish}  when $\tilde g_{t,n}=\tilde{f}_{t,n}^{(N)}$ is bounded above by $\gamma h_n^{-1}|K|(H_n/h_n)$ for any $t\in \R$, and this quantity is of order $O\big((h_n/H_n)^{\eta-1} H_n^{-1}\big)=o\big(n^{-1/2}\big)$ for some $\eta>2$ by the decay of $|K|$, $h_n$ and $H_n^{-1}$. To analyse \eqref{eqbiasmain} when $g_{t,n}=\tilde{f}_{t,n}^{(N)}$ we note that for $n$ large enough $\tilde{f}_{t,n}^{(N)}=\1_{(-\infty, \min\{u,-\varepsilon_n\}]}+\1_{[\min\{u,\varepsilon_n\}, u]}$, where $u=u(t,n)$ is defined in \eqref{tildeftnandutnbias}. Inputting this into \eqref{eqbiasmain} gives
	\begin{align*}
		- \gamma \Big( K_{h_n}\big(\min\{u,-\varepsilon_n\}\big) + K_{h_n}(u)- K_{h_n}\big(\min\{u,\varepsilon_n\}\big) \Big) - \lambda \int_{\substack{\quad \\ (-\varepsilon_n,\varepsilon_n)^{C}}} \1_{(-\infty, u]} (x) K_{h_n }(x) \, dx \\
		+\int_{-\infty}^u \nu \ast K_{h_n} \, (x) \, dx \, - \,  \sum_{\mathclap{j\in  \Z\setminus\{0\}}} \,  q_j  \int_{\substack{\quad \\\min\{u,-\varepsilon_n\}}}^{\substack{\min\{u,\varepsilon_n\} \\ \quad }} \qsub K_{h_n }(x -  j) \, dx + \int_{\R} K(y) \int_{\substack{\quad \\\min\{u,-\varepsilon_n\} -h_n y}}^{\substack{\min\{u,\varepsilon_n\} -h_n y \\ \quad }} \qsub \nu_{ac} (x) \, dx \, dy.
	\end{align*}
	We now bound the absolute value of all these terms but the first in the second line by quantities independent of $t$ and of order $o\big(n^{-1/2}\big)$. By the symmetry of $K$ and the decay of $|K|$, the absolute value of the first summand is bounded by $\gamma h_n^{-1} |K|(\varepsilon_n/h_n)|=O(h_n^{\eta-1} \varepsilon_n^{-\eta})$ for some $\eta>2$ and, since $h_n\!\sim\!\exp(-n^{\vartheta_h})$, $\varepsilon_n\!\sim \!\exp(-n^{\vartheta_\varepsilon})$ and $\vartheta_h \!\geq\! 2 \vartheta_\varepsilon$, it is $O(\exp( n^{\vartheta_\varepsilon (2-\eta)}))$ and thus of the required order. These arguments and  $\sup_n\sup_{t\in \R} \sup_{x\in \R} |\tilde{f}_{t,n}^{(N)}(x)|\!\leq\! 1$ justify that the absolute value of second summand is bounded by $2 \lambda \int_{\varepsilon_n/h_n}^{\infty} |K|\!=\! O\big((h_n/\varepsilon_n)^{\eta-1}\big)\!=\!o\big(n^{-1/2}\big)$. Furthermore, for any $t\!\in \!\R$ the absolute value of the second summand in the second line is bounded by
	\begin{equation*}
		\sum_{j\in  \Z\setminus\{0\}}   q_j  \int_{ (-\varepsilon_n- j)/h_n}^{(\varepsilon_n- j)/h_n }  |K|(x) \, dx \lesssim h_n^{\eta-1} \sum_{j\in  \Z\setminus\{0\}}   q_j  j^{1-\eta}= o\big(n^{-1/2}\big).
	\end{equation*}
	Due to Assumption \ref{enum:th:ass1b} and $K\in L^1(\R)$, for any $t\in \R$ the absolute value of the last summand is bounded by 
	\begin{equation*}
		\int_{\R} K(y) \int_{-\varepsilon_n -h_n y}^{\varepsilon_n -h_n y } \nu_{ac} (x) \, dx \, dy \lesssim ( \log (\varepsilon_n^{-1}))^{-\alpha}= o\big(n^{-1/2}\big),
	\end{equation*}
	where the last equality follows because  $\varepsilon_n\sim \exp(-n^{\vartheta_\varepsilon})$ with $\vartheta_\varepsilon>1/(2\alpha)$. We are therefore left with proving that the supremum over $t\in \R$ of the absolute value of
	\begin{align}\label{eqbiast}
		\int_{-\infty}^u \nu \ast K_{h_n} - N(t) = & \sum_{j\in  \Z\setminus\{0\}}  q_j  \bigg(\int_{-\infty}^u  K_{h_n}(x-  j) \, dx - \1_{(-\infty, t]}(  j) \bigg) \notag \\
		& + \int_{\R} K(y) \bigg( \int_{-\infty}^{u -h_n y }  \nu_{ac} (x) \, dx  - \int_{-\infty}^t \nu_{ac} (x) \, dx \bigg) dy 
	\end{align}
	is of the desired order, where to arrive to the last summand we have used that $\int_{\R} K=1$. Note that the quantity in brackets in the infinite sum is
	\begin{equation*}
		\int_{-\infty}^{(u- j)/h_n } K(x) \, dx - \1_{(-\infty, t]}(  j) = \left\{ \begin{array}{ll}
			\int_{(u- j) /h_n}^{\infty} K & \mbox{if } t\geq   j, \\
			& \\
			\int^{(u- j) /h_n}_{-\infty} K & \mbox{if } t<   j. \\
		\end{array}   
		\right.
	\end{equation*}
	For any $t\in \R$ and in view of \eqref{tildeftnandutnbias} we then have that the absolute value of the first summand in \eqref{eqbiast} is bounded by
	\begin{equation*}
		\sum_{j\in  \Z\setminus\{0\}} \,  q_j  \bigg(\int_{-\infty}^{-\varepsilon_n/h_n}  |K|(x) \, dx + \int^{\infty}_{\varepsilon_n/h_n}  |K|(x) \, dx  \bigg) = O\big( (h_n/\varepsilon_n)^{\eta-1}\big) = o\big(n^{-1/2}\big).
	\end{equation*}
	Using Assumption \ref{enum:th:ass1b}, the integrability of $K$, the decay of $|K|$, expression \eqref{tildeftnandutnbias} and the fact that $\int_{\R}\nu_{ac}\leq \lambda$, we conclude that the absolute value of the second summand in \eqref{eqbiast} is bounded, up to constants independent of $t$ and $n$, by
	\begin{align*}
		\int_{-\varepsilon_n/h_n}^{\varepsilon_n/h_n}  |K|(y) (\log(|u-h_n y &-t|^{-1}))^{-\alpha} dy + \lambda \int_{(-\varepsilon_n/h_n, \varepsilon_n/h_n)^{C}}  |K|(y)\, dy \\ &\lesssim (\log(\varepsilon_n^{-1}))^{-\alpha} +  (h_n/\varepsilon_n)^{\eta-1}  = o\big(n^{-1/2}\big),
	\end{align*}
	where the constant hidden in the notation $\lesssim$ is clearly independent of $t$ because $|u-t|\leq \varepsilon_n$ for all $t\in \R$ and the equality follows by the same arguments as above. We have then just shown that $\sup_{t\in \R} |B^{(N)}_{t,n}|=o\big(n^{-1/2}\big)$.
	
	For the last statement of the lemma regarding the linear combination $B_n$ note that
	\begin{equation*}
		\sum_{ j \in \Z\setminus\{0\}} C_{(j)} \big(  (B^{(q_j)}_n+q_j) \1_{|j| \leq \widetilde H_n }  - q_j\big) = \sum_{ \substack{|j| \leq \widetilde H_n \\ j \neq 0}} C_{(j)} B^{(q_j)}_n - \sum_{|j| > \widetilde H_n} C_{(j)} q_j.
	\end{equation*}
	Due to $\widetilde{H}_n\to \infty$, when finitely many coefficients $C_{(j)}$ are not zero $B_n$ has the required order because the individual quantities featuring in its expression do. For the same reason, under the alternative assumptions of (b*) we only have to analyse the last display to find the order of $B_n$. In view of the $j$-independent upper bound for $|B^{(q_j)}_n|$ in \eqref{eq|Bj|bias}, the absolute value of the first sum is bounded, up to constants independent of $n$, by
	\begin{equation*}
		\sup_{j\in \Z\setminus\{0\}} |C_{(j)}| \,\, \widetilde{H}_n \, \big( h_n^{\eta-1}  +  (h_n/\varepsilon_n)^{\eta-1} + (\log(\varepsilon_n^{-1}))^{-\alpha}\big) =o\big(n^{-1/2}\big),
	\end{equation*}
	where the last equality is justified because $\widetilde{H}_n\sim n^{\vartheta_{\widetilde{H}}}$, $h_n\sim \exp(n^{-\vartheta_{h}})$ and $\varepsilon_n\sim \exp(n^{-\vartheta_\varepsilon})$, with $\vartheta_h \geq 2 \vartheta_\varepsilon>\vartheta_\varepsilon$ and $\vartheta_{\varepsilon}>(1+2\vartheta_{\widetilde{H}})/(2\alpha)$. The absolute value of the second sum can be bounded, up to constants independent of $n$, by
	\begin{equation*}
		\sup_{j\in \Z\setminus\{0\}} |C_{(j)}|  \int_{ [-\widetilde{H}_n, \widetilde{H}_n]^{C}} \nu(dx) 
		\lesssim \widetilde{H}_n^{-\beta} \int_{\substack{\quad \\ [-\widetilde{H}_n, \widetilde{H}_n]^{C}}} |x|^{\beta} \nu(dx),
	\end{equation*}
	and this has order $o\big(n^{-1/2}\big)$ because $\widetilde H_n \sim n^{\vartheta_{\widetilde H}}$ for some $\vartheta_{\widetilde H}\geq 1/(2\beta)$ by assumption. Therefore, the infinite sum in $B_n$ is of order $o\big(n^{-1/2}\big)$ and so is $B$ under the second set of assumptions because we analysed \eqref{eqbiastailsvanish} when these hold. 
\end{proof}

Lastly, the following result guarantees joint convergence of finitely and infinitely many one-dimensional estimators and provides an explicit representation of the covariance of the asymptotic distributions. It follows immediately from the preceding results, and Propositions \ref{Proplambdamu} and \ref{Proppsqs} are corollaries of it.

\begin{thm}\label{ThmCLTEstimators}
	
	Adopt the setting and notation of Sections \ref{SecDefNot} and \ref{SecSettingEstimators}, and assume the finite measure $\nu$ satisfies Assumption \ref{enum:th:ass1} for some $\alpha>4$. Suppose $K$ satisfies \eqref{conditionsK} and take $h_n\sim \exp(n^{-\vartheta_h})$, $\varepsilon_n\sim \exp(n^{-\vartheta_\varepsilon})$ and $H_n\sim \exp(n^{\vartheta_H})$ such that $1/\alpha<2\vartheta_\varepsilon \leq \vartheta_h<1/4$ and $0<\vartheta_H <\vartheta_h $.  For any finite set $\mathcal{T}\subset \R$, any $C_{(\lambda)}, C_{(\gamma)}, C_{(j)}, C_t\in \R$, $j\in \Z\setminus\{0\}$ and $t\in \mathcal{T}$, and any $\widetilde{H}_n\to \infty$ we define the linear combination
	\begin{equation*}
		\widehat \varUpsilon_n\!:=\!\sqrt{n}  \Big( \!C_{(\lambda)} \big(\hat{\lambda}_n - \lambda\big)  + C_{(\gamma)} h_n^{-1}\big(\hat{\gamma}_n - \gamma\big) +\!\!\sum_{ j \in \Z\setminus\{0\}} \!\!C_{(j)} \big( \hat q_{j,n} \1_{|j| \leq \widetilde H_n }  - q_j\big) +  \sum_{t\in \mathcal{T}} C_t \big(\widehat{N}_n(t)-N(t)\big) \!\Big).
	\end{equation*}
	Furthermore, assume that either 
	\begin{enumerate}
		\item finitely many  $C_{(j)}$ are not zero, Assumption \ref{enum:th:ass2} holds for some $\beta>2$ and $\vartheta_H \geq 1/(2\beta)$;
		\item or $\sup_{j\in \Z\setminus\{0\}} \!|C_{(j)}| <\!\infty$, $\int_{\R}\! |x|^\beta \nu(dx)\!<\!\infty$ for some $\beta\!>\!1$, Assumption \ref{enum:th:ass1a} is satisfied for some $\alpha\!>\!8\beta/(\beta-1)$, $\widetilde{H}_n\!\sim\! n^{\vartheta_{\widetilde{H}}}$ for some $\vartheta_{\widetilde{H}}\!\in\! [1/(2\beta), 1/2)$, $\vartheta_h \!<\! (1-2\vartheta_{\widetilde{H}})/4$ and $\vartheta_\varepsilon \!>\!1/\alpha$.
	\end{enumerate}
	Then we have that
	\begin{align}\label{lincombCLT}
		\widehat \varUpsilon_n \to^d \! N\big(0,\sigma^2 \big) \qquad \mbox{as } n\to \infty
	\end{align}
	where $\sigma^2$ is finite and satisfies
	\begin{align*} 
		{\Deltaup}^2\sigma^2 \!=\!  \int_{\R} \! \left( l \ast \FII{ \varphi^{-1}(-\cdot)} \! (x) \right)^2 \! P(dx), \quad \! \mbox{where} \!\quad l \! := C_{(\lambda)} f^{(\lambda)} \! +\! \! \!\sum_{ j \in \Z\setminus\{0\}} \! \!\! C_{(j)}  f^{(q_j)} \!+  \sum_{t\in \mathcal{T}} C_t f^{(N)}_{t}\!.
	\end{align*}
\end{thm}

\begin{proof}
	Notice that the sum of distinguished logarithms equals the distinguished logarithm of the product. Then, in view of the expressions of the estimators $\hat{\lambda}_n$, $\hat{\gamma}_n$, $\hat q_{j,n}$ and $\widehat{N}_n(t)$ in Section \ref{SecSettingEstimators} and the expressions for $B^{(\lambda)}_n, B^{(\gamma)}_n, B^{(q_j)}_n$ and $B^{(N)}_{t,n}$ in Lemma \ref{Lemmabiasterms} we can write
	\begin{equation} \label{eqLinCombEstimators}
		\widehat \varUpsilon_n= \sqrt{n} \, \frac{1}{{\Deltaup}} \int_{\R} f_{t,n}(x)\, \FII{\Log \frac{\varphi_n}{\varphi} \FT K_h} (x) \, dx + \sqrt{n} \, B_n,
	\end{equation}
	where $f_{t,n}$ and $B_n$ are defined in \eqref{defg} and \eqref{eqbias}, respectively. Note that all the assumptions of Lemma \ref{Lemmabiasterms} are trivially satisfied by those assumed here as the latter include the former. Therefore the second summand in the last display is negligible as $n\to \infty$. To find the asymptotic distribution of the first summand we use Theorem \ref{ThmGenericCLT} applied to $f_{t,n}$ as defined in \eqref{defg}. We can do this in view of Lemma \ref{Lemmaglambdagnu} because its assumptions are trivially satisfied by those assumed here. Then, using the last conclusion of the lemma, it follows that
	\begin{equation*}
		\widehat \varUpsilon_n \to^d N(0,\tilde{\sigma}^2),
	\end{equation*}
	where $\tilde{\sigma}^2$ satisfies the same expression as $\sigma^2$ after \eqref{lincombCLT} for 
	\begin{equation*}
		\tilde l:= C_{(\lambda)} f^{(\lambda)} + \sum_{j\in \Z \setminus\{0\}} C_{(j)} f^{(q_j)} + \sum_{t\in \mathcal{T}} C_t \Big( f^{(N)}_{t} - \frac{1}{2} \1_{\{t\}} \1_{\R \setminus   \Z} (t) \Big)
	\end{equation*}
	in place of $l$. Since $\tilde l(0)=0$ the variance $\tilde{\sigma}^2$ has the same expression as \eqref{LimCov} when $\tilde{g}_t^{1)}=\tilde{g}_t^{2)}= {\Deltaup}^{-1} \tilde l$. Furthermore, the third summand in $\tilde l$ agrees with the third summand in $l$ up to a zero Lebesgue-measure set disjoint from $  \Z$. Hence, the last claim in Lemma \ref{Lemmalimsints} guarantees $\tilde{\sigma}^2=\sigma^2$ and the finiteness of this quantity follows by the conclusions of Theorem \ref{ThmGenericCLT}. 
\end{proof}

\subsection{Proof of Propositions \ref{Proplambdamu} and \ref{Proppsqs}}\label{SecProofProps} 

Proposition \ref{Proplambdamu} follows immediately from Theorem \ref{ThmCLTEstimators} by taking $\mathcal{T}=\emptyset$, $C_{(j)}=0$ for all $j\in \Z\setminus\{0\}$ and $C_{(\lambda)}=1, C_{(\gamma)}=0$ or $C_{(\lambda)}=0, C_{(\gamma)}=1$ when estimating $\lambda$ or $\gamma$, respectively. The conclusion of Proposition \ref{Proppsqs} on estimating $q_j$, $j\in \Z \setminus\{0\}$, follows analogously. To show the conclusion on $p_j$ we write
\begin{equation*}
	\sqrt{n} \, \big( \hat p_{j,n} - p_j  \big) = \hat \lambda_n^{-1} \,  \sqrt{n} \, \Big( \big(\hat q_{j,n} - q_j \big) + p_j \big(\lambda - \hat \lambda_n \big) \Big).
\end{equation*}
In view of Theorem \ref{ThmCLTEstimators} with $C_{(\lambda)}=-p_j, C_{(\gamma)}= C_l=0$ for all $l\in \Z \setminus\{0,j\}$, $\mathcal{T}=\emptyset$ and $C_{(j)}=1$, the quantity by which $\hat \lambda_n^{-1}$ is multiplied on the right hand side converges to $N\big(0, \lambda^2\sigma_{p_j}^2\big)$. Since $\hat \lambda_n$ converges to $\lambda$ (constant) by Theorem \ref{ThmCLTEstimators}, the conclusion follows by Slutsky's lemma.


\subsection{Proof of Theorem \ref{ThmNF}} \label{SecProofNF} 

In view of the expressions for $\widehat{F}_n(t)$ and $F(t)$ in Section \ref{SecSettingEstimators}, for any $t\in \R$ we can write
\begin{align*}
	\sqrt{n} \, \big(  \widehat{F}_n(t) - F(t)\big) = \hat \lambda_n^{-1} \, \sqrt{n} \, \Big( \big(\widehat{N}_n(t) - N(t) \big) + F(t) \big( \lambda - \hat \lambda_n \big) \Big).
\end{align*}
Therefore, if we show that the quantity by which $\hat \lambda_n^{-1}$ is multiplied converges in distribution in $\ell^{\infty}(\R)$ to $\lambda \mathbb{G}^{F}$ and that this limit is tight, the result for $F$ follows by Slutsky's lemma (cf. Example 1.4.7 in \cite{AvdVW96}) because $\hat \lambda_n\to \lambda$ (constant) in distribution in $\R$ in view of Proposition \ref{Proplambdamu}. To show result for $N$, we need to work with the first summand on the right hand side. Consequently, we can unify the proof of the two results by considering the quantity
\begin{equation*}
	\widehat{G}_n(t) := \big(\widehat{N}_n(t) - N(t) \big) + \varsigma F(t) \big( \lambda - \hat \lambda_n \big), \quad t\in \R,
\end{equation*}
where $\varsigma=0$ or $1$ depending on whether we are estimating $N$ or $F$, respectively. With this in mind we define $f^{(G)}_{t,n}:=f^{(N)}_{t,n}- \varsigma F(t) f^{(\lambda)}_n, t\in \R$, where $f^{(\lambda)}_n$ is introduced in Section \ref{SecSettingEstimators}.

Note that the assumptions in the first part of Lemma \ref{Lemmaglambdagnu} are included in those assumed here. Therefore, the assumptions of Theorem \ref{ThmGenericCLT} are satisfied, and in view of the expression of the estimator $\widehat{N}_n(t)$ in Section \ref{SecSettingEstimators}, the expressions for $B^{(N)}_{t,n}$ and $B^{(\lambda)}_{n}$ in Lemma \ref{Lemmabiasterms}, and the properties of the distinguished logarithm, we can write
\begin{align*}
	\sqrt{n} \, \widehat{G}_n(t)  & = \sqrt{n} \, \frac{1}{{\Deltaup}} \int_{\R} f^{(G)}_{t,n}(x) \, \FII{\Log \frac{\varphi_n}{\varphi} \FT K_h} (x) \, dx + \sqrt{n} \, \big(B^{(N)}_{t,n} - \varsigma F(t) B^{(\lambda)}_n\big)\\
	&=^{\Pr}  \sqrt{n} \, \frac{1}{{\Deltaup}} \int_{\R} f^{(G)}_{t,n} \ast K_{h_n} \ast \FII{\varphi^{-1}(-\cdot)}(x) \big( P_n -P \big)(dx) \\
	& \qquad  + \sqrt{n} \, \frac{1}{{\Deltaup}} \int_{\R} \FT f^{(G)}_{t,n} (-u) R_n(u) \FT K_{h_n}(u) du + \sqrt{n} \, \big(B^{(N)}_{t,n} - \varsigma F(t) B^{(\lambda)}_n\big),
\end{align*}
where $R_n$ is as in Theorem \ref{ThmLinearisation}. We are showing a central limit theorem under the uniform norm, so we now argue that the supremum over $t\in \R$ of the last line vanishes as $n\to \infty$ in the same sets of $\Pr$-probability approaching $1$ in which the last equality holds. By Lemma \ref{Lemmabiasterms} we have $\sup_{t\in \R}|B^{(N)}_{t,n}| = o\big(n^{-1/2}\big)$ and $|B^{(\lambda)}_n| = o\big(n^{-1/2}\big)$, so the last summand vanishes as $n\to \infty$ in view of $\sup_{t\in \R}F(t)\leq 1$. Due to $\supp(\FT K_{h}) \subseteq [-h^{-1},h^{-1}]$, the supremum of the first summand in the last line is bounded by
\begin{equation*}
	\norm{K}_1 \, {\Deltaup}^{-1} \, \sqrt{n}   \sup_{|u|\leq h_n^{-1}} |R_n(u)| \, \sup_{t\in \R} \, \int_{-h_n^{-1}}^{h_n^{-1}} \big(|\FT f^{(N)}_{t,n} (u)| + \varsigma F(t) |\FT f^{(\lambda)}_n (u)| \big)  du = o_{\Pr} \big( 1\big),
\end{equation*}
where the equality follows from Theorem \ref{ThmLinearisation} and Lemma \ref{Lemmaglambdagnu}. Consequently we have to show the corresponding functional central limit theorem for the linear term 
\begin{equation*}
	\sqrt{n}\,  (P_n -P)\psi_{t,n}:=\sqrt{n} \int_{\R} \psi_{t,n} (x) (P_n - P)(dx),
\end{equation*}
where 
\begin{equation*}
	\psi_{t,n}:= {\Deltaup}^{-1} f^{(G)}_{t,n} \ast K_{h_n} \ast \FII{\varphi^{-1}(-\cdot)}.
\end{equation*}
This type of result follows by showing a central limit theorem (for triangular arrays in this case) for the finite-dimensional distributions and tightness of the limiting process. Conveniently, Theorem 2.11.23 in \cite{AvdVW96} explicitly gives assumptions under which these follow and using it adds clarity to the proofs. Thus, we recall it adapted to our needs and refer the reader to \cite{AvdVW96} for the concepts in it such as envelope functions, outer measures, bracketing numbers and entropies, etc.

\begin{thm} \label{ThmTightness}
	For each $n$, let $\varPsi_n:= \{\psi_{t,n}:t\in \R \}$ be a class of measurable functions indexed by a totally bounded semimetric space $(\R, \rho)$. Given envelope functions $\Psi_n$ assume that
	\begin{align*}
		P^* \Psi_n^2 = O(1), \qquad P^* \Psi_n^2 \mathds{1}_{\{ \Psi_n > \kappa \sqrt{n}  \}} \to 0 \qquad \mbox{for every } \kappa>0,
	\end{align*}
	\begin{align*}
		\sup_{\mathclap{\rho(s,t)<\delta_n}} \, \, P(\psi_{s,n} - \psi_{t,n})^2 \to 0 \, \, \,  \mbox{ and } \, \,  \int_0^{\delta_n} \sqrt{ \log N_{[\, ]} (\epsilon \norm{\Psi_n}_{P,2}, \varPsi_n, L_2(P))} \, d\epsilon \to 0  \quad \mbox{for every } \delta_n \downarrow 0, 
	\end{align*}
	where $\norm{\psi}_{L_2(P)}=(\int_{\R} |\psi|^2 P)^{1/2}$. Then the sequence of stochastic processes
	\begin{equation*}
		\left\lbrace  \sqrt{n}\,  (P_n -P) \,\psi_{t,n}: t \in \R \right\rbrace 
	\end{equation*}
	is asymptotically tight in $l^{\infty}(\R)$ and converges in distribution to a tight Gaussian process provided the sequence of covariance functions $P \psi_{s,n} \psi_{t,n} - P \psi_{s,n} P \psi_{t,n}$ converges pointwise on $\R\times \R$.
\end{thm}

We first compute the pointwise limit of the covariance functions. Due to the limiting distribution being a tight (centred) Gaussian process, the former limit uniquely identifies the process. This allows us to identify the so-called intrinsic covariance semimetric of the Gaussian process and, as it is customary, we take $\rho$ equal to this semimetric (see the second half of Chapter 1.5 and Chapter 2.1.2 of \cite{AvdVW96} for a discussion on this and other choices). We then check the remaining assumptions of the theorem in order of appearance. 
\mbox{ }\\

\textit{Convergence of the covariance functions.} Note that the assumptions here include those of Lemma \ref{Lemmalimsints} and of the first part of Lemma \ref{Lemmaglambdagnu}. For any $s,t \in \R$ fixed, $P \psi_{s,n} \psi_{t,n}$ has the form of \eqref{LimCov} when $g^{1)}={\Deltaup}^{-1} f_{s,n}^{(G)}$ and $g^{2)}={\Deltaup}^{-1} f^{(G)}_{t,n}$ and we can use the lemmata to conclude that
\begin{align*}
	\lim_{n\to \infty} P \psi_{s,n} \psi_{t,n} = \frac{1}{{\Deltaup}^2} \int_{\R} & \big(l_s-\varsigma F(s) f^{(\lambda)} \big) \ast \FII{\varphi^{-1}(-\cdot)} \, (x) \\
	& \times \big(l_t-\varsigma F(t) f^{(\lambda)} \big) \ast \FII{\varphi^{-1}(-\cdot)} \, (x) \, P(dx)
\end{align*}
and that this is finite, where $l_t$ is defined in \eqref{Limftn}. Furthermore, $l_t-\varsigma F(t)f^{(\lambda)}$ agrees with $f_t^{(G)}:=f^{(N)}_{t}-\varsigma F(t)f^{(\lambda)}$ up to a zero Lebesgue-measure set disjoint from $  \Z$ for any $t\in \R$. Hence, the last claim in Lemma \ref{Lemmalimsints} guarantees the last display equals
\begin{align*}
	\frac{1}{{\Deltaup}^2} \int_{\R} & f_s^{(G)}\ast \FII{\varphi^{-1}(-\cdot)} \, (x) \, f^{(G)}_{t} \ast \FII{\varphi^{-1}(-\cdot)} \, (x) \, P(dx).
\end{align*}
Conclusion \eqref{LimMean} and the fact that $\big(l_t-\varsigma F(t)f^{(\lambda)}\big)(0)=0$ for any $t\in \R$ justify that $P \psi_{t,n}=0$ for all $t\in \R$ and therefore the sequence of covariance functions $P \psi_{s,n} \psi_{t,n} - P \psi_{s,n} P \psi_{t,n}$ converges pointwise to $(1-\varsigma )\Sigma_{s,t}^{N} + \varsigma \lambda^2 \Sigma_{s,t}^{F}$.
\mbox{ }\\

\textit{ Total boundedness of $\R$ under the internal covariance semimetric $\rho$.}  In view of the limiting covariance we just computed we take 
\begin{equation*}
	\rho(s,t)= {\Deltaup}^{-1} \left(P\Big(f_s^{(G)} \ast \FII{\varphi^{-1}(-\cdot)} - f^{(G)}_{t} \ast \FII{\varphi^{-1}(-\cdot)}\Big)^2\right)^{1/2}, \qquad s,t\in \R.
\end{equation*}
To show that $\R$ is totally bounded under this semimetric we bound this expression by another semimetric under which $\R$ is totally bounded. Due to $\FII{\varphi^{-1}(-\cdot)} $ being a finite measure and $\sup_{t\in \R}\sup_{x\in \R}|f^{(G)}_{t}(x)|\leq 2$, Minkowski's inequality for integrals guarantees that
\begin{align*}
	\rho(s,t)^2 & \lesssim  {\Deltaup}^{-2} P \, \left|\left(f_s^{(G)} -f_t^{(G)}\right)\ast \FII{\varphi^{-1}(-\cdot)} \right|  \leq  {\Deltaup}^{-2} P \left|f_s^{(G)} -f_t^{(G)}\right|\ast \left|\FII{\varphi^{-1}(-\cdot)} \right|,
\end{align*}
where the last inequality follows from Jordan's decomposition of finite measures and $|\FII{\varphi^{-1}(-\cdot)} |$ is the positive measure given by the sum of the positive and negative parts of $\FII{\varphi^{-1}(-\cdot)}$ in this decomposition. Finally, note that
\begin{align*}
	P  \left|f_s^{(G)} -f^{(G)}_{t}\right|\ast \left|\FII{\varphi^{-1}(-\cdot)} \right| &=  \left|f_s^{(G)} -f^{(G)}_{t}\right|\ast \left|\FII{\varphi^{-1}(-\cdot)} \right| \ast \bar P \, (0) \\
	& \leq \left|f_s^{(N)} -f^{(N)}_{t}\right|\ast \left|\FII{\varphi^{-1}(-\cdot)} \right| \ast \bar P \, (0) \\
	& \quad +\varsigma \left|F(s)-F(t)\right| f^{(\lambda)} \ast \left|\FII{\varphi^{-1}(-\cdot)} \right| \ast \bar P \, (0) \\
	& =  \mu_0\big( (\min\{s,t\}, \max\{s,t\}]\big),
\end{align*}
where $\bar P (A) = P(-A)$ and $\bar \mu (A) = \mu(-A)$ for any Borel $A\subseteq \R$, $\mu:= \left|\FII{\varphi^{-1}(-\cdot)} \right| \ast \bar P$ and $\mu_0:= \bar \mu-\bar \mu(\{0\})\delta_0+\varsigma \nu$. The conclusion then follows because $ \mu_0$ is a finite measure on $\R$.
\mbox{ }\\

\textit{ Conditions on the envelope functions of $\varPsi_n:= \{ \psi_{t,n}: t\in \R \}$.}  Note that $$\sup_n\sup_{t\in \R} \sup_{x\in\R}|f^{(G)}_{t,n}(x)|\leq 2, \quad \norm{K_h}_1=\norm{K}_1<\infty$$ and that $\FII{\varphi^{-1}(-\cdot)}$ is a finite measure. Using Minkowski's inequality for integrals we have that $ \sup_n \sup_{t\in \R} \sup_{x\in \R} |\psi_{t,n} (x)| \leq \Psi$ for some $\Psi\in (0, \infty)$ and we can take $\Psi_n=\Psi$ for all $n$. The two conditions on the envelope functions then follow immediately.
\mbox{ }\\

\textit{ Control of $P(\psi_{s,n}-\psi_{t,n})^2$.} In the following we repeatedly use the fact that if $f, g$ are bounded functions and $\mu$ is a finite positive measure then
\begin{equation}\label{eqSquaresCS}
	\mu (f+g)^2:= \int_{\R} (f(x)+g(x))^{2} \mu (dx) \leq 2 \left( \mu f^2+\mu g^2\right),
\end{equation} 
and hence to control the left hand side we can control $\mu f^2$ and $\mu g^2$ separately. Therefore, writing 
\begin{align*}
	\psi_{s,n}-\psi_{t,n} =  & \, \big( \psi_{s,n}- {\Deltaup}^{-1} f_s^{(N)}\ast \FII{\varphi^{-1}(-\cdot)} \big)- \big(\psi_{t,n} - {\Deltaup}^{-1} f^{(N)}_{t}\ast \FII{\varphi^{-1}(-\cdot)}\big) \\
	& - {\Deltaup}^{-1} \varsigma \big( F(s)- F(t) \big) \big( f^{(\lambda)}_n\ast K_{h_n} -f^{(\lambda)}\big) \ast  \FII{\varphi^{-1}(-\cdot)} \\
	& + {\Deltaup}^{-1} (f_s^{(G)}-f^{(G)}_{t})\ast \FII{\varphi^{-1}(-\cdot)}
\end{align*}
and noting that $\sup_{s,t \in \R} \big( F(s)- F(t) \big)^{2} \lesssim 1$, to control $P(\psi_{s,n}-\psi_{t,n})^2$ we only need to control
\begin{equation*}
	P \big(\psi_{t,n} - {\Deltaup}^{-1} f^{(N)}_{t}\ast \FII{\varphi^{-1}(-\cdot)} \big)^2, \quad P \Big( \big( f^{(\lambda)}_n\ast K_{h_n} -f^{(\lambda)}\big) \ast  \FII{\varphi^{-1}(-\cdot)} \Big)^2 \quad \mbox{ and } \quad\rho(s,t).
\end{equation*}
We analyse them in reverse order. The behaviour of the last term when $\rho(s,t)<\delta_n \downarrow 0$ is trivial. By Lemma \ref{Lemmaglambdagnu} the function $f^{(\lambda)}_n\ast K_{h_n} -f^{(\lambda)}$ converges to $0$ pointwise as $n\to \infty$. By the boundedness of this function, the finiteness of $\FII{\varphi^{-1}(-\cdot)} $ and Minkowski's inequality for integrals we can argue as when proving Lemma \ref{Lemmalimsints}, and dominated convergence guarantees the limit as $n\to \infty$ of the second quantity in the last display equals $0$ regardless of the sequence $\delta_n$. Consequently, if we show that the supremum over $t\in\R$ of the first quantity vanishes as $n\to \infty$, then the first term in the second display of Theorem \ref{ThmTightness} also vanishes in the limit. Note that in view of the expressions for $f^{(N)}_{t,n}$ and $\tilde{f}^{(N)}_{t,n}$ in \eqref{defftn} and \eqref{tildeftnandutnbias}, respectively, we can write
\begin{align*}
	\psi_{t,n} - {\Deltaup}^{-1} f^{(N)}_{t}\ast \FII{\varphi^{-1}(-\cdot)} = &\,{\Deltaup}^{-1} ( \tilde{f}^{(N)}_{t,n}\ast K_{h_n}- f^{(N)}_{t})\ast \FII{\varphi^{-1}(-\cdot)} \\
	&- {\Deltaup}^{-1} \1_{[-H_n, H_n]^{C}} \tilde{f}^{(N)}_{t,n} \ast K_{h_n} \ast \FII{\varphi^{-1}(-\cdot)}.
\end{align*}
Then, using \eqref{eqSquaresCS} and \eqref{LimglambdaHn}, together with the finiteness of $\FII{\varphi^{-1}(-\cdot)}$ and $P$, we only need to analyse the supremum over $t\in\R$ of
\begin{align}\label{eqL2Ptildeftnmlt}
	P \Big(( \tilde{f}^{(N)}_{t,n}\ast K_{h_n} \!-\! f^{(N)}_{t})\ast & \FII{\varphi^{-1}(-\cdot)} \! \Big)^2 \notag \\
	=  & P_d \Big(( \tilde{f}^{(N)}_{t,n}\ast K_{h_n} \!-\! f^{(N)}_{t})\ast \varPhi_d + ( \tilde{f}^{(N)}_{t,n}\ast K_{h_n} \!-\! f^{(N)}_{t})\ast \varPhi_{ac} \Big)^2  \\
	&+ \! P_{ac}\Big( ( \tilde{f}^{(N)}_{t,n}\ast K_{h_n} \!-\! f^{(N)}_{t})\ast \varPhi_d + ( \tilde{f}^{(N)}_{t,n}\ast K_{h_n} \!-\! f^{(N)}_{t})\ast \varPhi_{ac}\Big)^2 \! \notag,
\end{align}
where $P_d$ and $\varPhi_d$ are discrete and $P_{ac}$ and $\varPhi_{ac}$ are absolutely continuous finite measures with respect to Lebesgue's measure, and, by the decomposition of $\nu$ in \eqref{LebDecNu} and Lemma 27.1 in \cite{S99}, 
\begin{equation*}
	P_d= \sum_{k=0} \nu_d^{\ast k} \frac{{\Deltaup}^k}{k!}, \quad P_{ac}=P-P_d, \quad \varPhi_d= \sum_{k=0} \bar \nu_d^{\ast k} \frac{(-{\Deltaup})^k}{k!} \quad \mbox{ and } \quad \varPhi_{ac}=\FII{\varphi^{-1}(-\cdot)}- \varPhi_d.
\end{equation*}
By \eqref{eqSquaresCS} we then have to control the four individual terms arising from \eqref{eqL2Ptildeftnmlt}. In view of Assumption \ref{enum:th:ass1a} we notice that $P_d$ may have atoms only at $  \Z$ and hence for the first term we need to analyse $( \tilde{f}^{(N)}_{t,n}\ast K_{h_n}- f^{(N)}_{t})\ast \varPhi_d\,( j)$ for any $j\in \Z$. For the second and last quantities we analyse $( \tilde{f}^{(N)}_{t,n}\ast K_{h_n}- f^{(N)}_{t})\ast \varPhi_{ac}\,(x)$ for any $x\in \R$. For the third quantity we have that, by Fubini's theorem and Jensen's inequality,
\begin{align*}
	P_{ac} \Big(( \tilde{f}^{(N)}_{t,n}\ast K_{h_n}- f^{(N)}_{t})\ast \varPhi_d \Big)^2  \!&= \! \int_{\R^2}  \int_{\R}  \! \big( \tilde{f}^{(N)}_{t,n}\ast K_{h_n}- f^{(N)}_{t}\big)(x-y_1) \\
	& \qquad \qquad \times \big( \tilde{f}^{(N)}_{t,n}\ast K_{h_n}- f^{(N)}_{t}\big)(x-y_2) \\
	& \qquad \qquad \quad \times P_{ac}(dx) \, \varPhi_d (dy_1) \, \varPhi_{d}(dy_2) \\
	& \leq \left( \bar{\varPhi}_d \big(( \tilde{f}^{(N)}_{t,n}\ast K_{h_n}- f^{(N)}_{t})^2 \ast \bar{P}_{ac} \big)^{1/2}\right)^2,
\end{align*}
where as usual $\bar{\varPhi}_d(A)=\bar{\varPhi}_d(-A) $ and $\bar{P}_{ac}(A)=\bar{P}_{ac}(-A)$ for any Borel set $A \subseteq \R$. Since $\varPhi_d$ may have atoms only at $  \Z$ by Assumption \ref{enum:th:ass1a}, to control this term we therefore require to analyse $( \tilde{f}^{(N)}_{t,n}\ast K_{h_n}- f^{(N)}_{t})^2 \ast \bar{P}_{ac}\, ( j)$ for all $j\in \Z$. However, noting that $\bar{P}_{ac}=\bar{\nu}_{ac} \ast \mu_1$ and $\varPhi_{ac}=\bar{\nu}_{ac} \ast \mu_2$ for some finite measures $\mu_1$ and $\mu_2$, to control all the terms in \eqref{eqL2Ptildeftnmlt} but the first we only have to analyse $( \tilde{f}^{(N)}_{t,n}\ast K_{h_n}- f^{(N)}_{t})^k\ast \bar{\nu}_{ac}\,(x)$ for all $x\in \R$ and $k=1,2$. The rest of the section is then devoted to showing that for some $\eta>2$ and $\alpha>0$
\begin{equation*}
	\sup_{t\in \R}\, \sup_{j\in \Z} |( \tilde{f}^{(N)}_{t,n}\ast K_{h_n}- f^{(N)}_{t})\ast \varPhi_d\,( j)| \lesssim \Big( \frac{h_n}{\varepsilon_n}\Big)^{\eta-1}  
\end{equation*}
and
\begin{equation*}
	\sup_{t\in \R} \,\sup_{x\in \R}|( \tilde{f}^{(N)}_{t,n}\ast K_{h_n}- f^{(N)}_{t})^k\ast \bar{\nu}_{ac}\,(x)| \lesssim (\log(\left| \varepsilon_n \right|^{-1}))^{-\alpha} + \Big(\frac{h_n}{\varepsilon_n}\Big)^{\eta-1},
\end{equation*}
from which it easily follows that the supremum over $t\in \R$ of \eqref{eqL2Ptildeftnmlt} vanishes as $n\to \infty$ because $h_n, \varepsilon_n\to 0$ with $h_n=o(\varepsilon_n)$. To bound the first quantity note that using the symmetry of $K$ we have that for any $j\in \Z$ and $t\in \R$
\begin{equation*}
	( \tilde{f}^{(N)}_{t,n}\ast K_{h_n}- f^{(N)}_{t})\ast \varPhi_d\,( j) = \sum_{l\in \Z} \bar \varPhi_d(\{ (l-j)\})  \left( \int_{\R} \tilde{f}^{(N)}_{t,n} ( l +x) K_{h_n}(x)\,dx - f^{(N)}_{t}( l)\right).
\end{equation*}
Without loss of generality assume $\varepsilon_n<1/2$ and recall the definition of $u(t,n)$ in \eqref{tildeftnandutnbias}. Using that $\int_{\R}K=1$ the quantity in brackets in the previous display can then be written as
\begin{align}\label{eqPerturbationsd}
	\int_{\R} & \left( \1_{(-\infty, u]} - \1_{(-\varepsilon_n, \varepsilon_n) } \1_{[0,\infty)} (t) \right) ( l + x) K_{h_n}(x) \, dx - \1_{(-\infty, t]} \1_{\R\setminus\{0\}} ( l) \notag \\
	&= \1_{(-\infty, t]} \1_{\R\setminus\{0\}} ( l) \int_{ (u- l)/h_n}^{\infty} K(x)\,dx+ \left( 1- \1_{(-\infty, t]} \1_{\R\setminus\{0\}} ( l)\right) \int^{(u- l)/h_n}_{-\infty}K(x)\,dx \notag \\
	& \quad - \1_{[0,\infty)} (t) \int_{(-\varepsilon_n- l)/h_n}^{(\varepsilon_n- l)/h_n}  K(x)\,dx.
\end{align}
If $t< 0$ or $t\geq 0$ we have that $u(t,n)\leq -\varepsilon_n$ or $u(t,n) \geq \varepsilon_n$, respectively, so when $l=0$ the absolute value of this display is bounded above by 
\begin{equation*}
	\int_{-\infty}^{-\varepsilon_n/h_n}  |K|(x)\,dx + \int^{\infty}_{\varepsilon_n/h_n}  |K|(x)\,dx \lesssim \Big( \frac{h_n}{\varepsilon_n}\Big)^{\eta-1}
\end{equation*}
for some $\eta>2$ using the decay of $|K|$ in \eqref{conditionsK}. When $l\neq 0 $ the absolute value of the third summand in \eqref{eqPerturbationsd} is also bounded by the last display. The other two summands are also bounded by the last display when $l\neq 0$ because if $ l \leq t$ then $u- l\geq  \varepsilon_n$ and if $ l > t$ then $u- l\leq  \varepsilon_n$. Due to the last display not depending on $t,j$ or $l$ and the fact that $\bar{\varPhi}_d$ is a finite measure we therefore conclude that 
\begin{equation*}
	\sup_{t\in \R}\, \sup_{j\in \Z} |( \tilde{f}^{(N)}_{t,n}\ast K_{h_n}- f^{(N)}_{t})\ast \varPhi_d\,( j)| \lesssim \Big( \frac{h_n}{\varepsilon_n}\Big)^{\eta-1}.
\end{equation*}
To bound the other quantity we note that, using $\int_{\R}K=1$, the symmetry of $K$ and the positivity of $\nu_{ac}$, we have that for $k=1,2$ and any $t,x\in \R$
\begin{align*}
	\left| ( \tilde{f}^{(N)}_{t,n}\ast K_{h_n}- f^{(N)}_{t})^k\ast \bar{\nu}_{ac}\,(x) \right| &\leq  \int_{\R} \left| \int_{\R}\big(\tilde{f}^{(N)}_{t,n}(x+y+h_n z)-f^{(N)}_{t}(x+y)\big) K(z) dz\right|^k \nu_{ac}(y)  dy\\
	& \leq \int_{\R} \int_{\R} \left| \tilde{f}^{(N)}_{t,n}(x+y+h_n z)-f^{(N)}_{t}(x+y)\right| \nu_{ac}(y)  dy \, |K|(z) dz,
\end{align*}
where in the last inequality we have used Jensen's inequality when $k=2$, Fubini's theorem and the fact that $\tilde{f}^{(N)}_{t,n}$ and $f^{(N)}_{t}$ only take values $0,\pm 1$. Note that, because $\nu_{ac}$ is absolutely continuous, the truncation of $f^{(N)}_{t}$ at the origin can be ignored and, similarly to above,
\begin{align*}
	\tilde{f}^{(N)}_{t,n}(x+y+h_n z)-\1_{(-\infty, t]} (x+y) = & \1_{(-\infty, u-x-h_n z]} (y) - \1_{(-\infty, t-x]} (y) \\
	&- \1_{(-\varepsilon_n-x-h_n z, \varepsilon_n-x-h_n z)}(y) \1_{[0,\infty)}(t).
\end{align*}
Therefore, if Assumption \ref{enum:th:ass1b} is satisfied for some $\alpha>0$, we have that for $k=1,2$ and any $t,x\in \R$
\begin{align*}
	\left| ( \tilde{f}^{(N)}_{t,n}\ast K_{h_n}\!-\! f^{(N)}_{t})^k\ast \bar{\nu}_{ac}\,(x) \right| & \!\lesssim \!\int_{\R} \!\min\{(\log(\left| u-t-h_n z \right|^{-1}))^{-\alpha},1\} |K|(z) dz \!+\! (\log(\left| \varepsilon_n \right|^{-1}))^{-\alpha} \\
	& \lesssim \!(\log(\left| \varepsilon_n \right|^{-1}))^{-\alpha} + (h_n/\varepsilon_n)^{\eta-1},
\end{align*}
where in the last inequality we have used that $|u-t|\leq \varepsilon_n$ and the decay of $|K|$ assumed in \eqref{conditionsK}. Recall that $h_n\sim \exp(-n^{\vartheta_h})$ and $\varepsilon_n\sim \exp(-n^{\vartheta_\varepsilon})$ with $\vartheta_h \geq 2\vartheta_\varepsilon$ and $\vartheta_\varepsilon>1/(2\alpha)$. We conclude that the supremum over $x\in\R$ and $t\in\R$ of the left hand side is also bounded by the right hand side by noting that the constants hidden in the notation $\lesssim$ are independent of them in view of the independence of the notation $\lesssim$ in Assumption \ref{enum:th:ass1b}.
\mbox{ }\\

\textit{Checking the bracketing entropy condition.} To check the remaining condition of Theorem \ref{ThmTightness} we first recall that $\Psi_n=\Psi$ is independent of $n$. Second, we claim that the classes $\varPsi_n$ are all contained in a single ball in the space of bounded variation functions. Assuming this, the bracketing entropy in the theorem is bounded above by the bracketing entropy of this ball and, by Corollary 3.7.51 in \cite{GN15}, the latter is bounded above by $(\epsilon \Psi)^{-1}$. Therefore the bracketing entropy condition follows if we prove the claim. 

In view of \eqref{defftn}, the definition of $u(t,n)$ in \eqref{tildeftnandutnbias} that we recalled above and by properties of the convolution, the weak derivative of $\psi_{t,n}$ is given by 
\begin{align*}
	\psi_{t,n}'={\Deltaup}^{-1}\Big( & \big(\delta_{-H_n}-\delta_{-\varepsilon_n} + \delta_{\varepsilon_n} -\delta_{H_n}\big) \1_{(-\infty, u]} - \1_{[-H_n,H_n]\setminus(-\varepsilon_n, \varepsilon_n)} \delta_{u} \\
	&- \varsigma F(t) \big(\delta_{-H_n}-\delta_{-\varepsilon_n} + \delta_{\varepsilon_n} -\delta_{H_n}\big) \Big) \ast K_{h_n}\ast \FII{\varphi^{-1}(-\cdot)}.
\end{align*}
Thus, using Minkowski's inequality for integrals and that $\norm{K_h}_{1}=\norm{K}_1$, we have that
\begin{equation*}
	\norm{\psi_{t,n}'}_{TV} \leq 9 {\Deltaup}^{-1}\norm{K}_1 \norm{\FII{\varphi^{-1}(-\cdot)}}_{TV} <\infty
\end{equation*}
and the claim follows by noting that the upper bound does not depend on $t$ or $n$. 

\subsection{Proof of Theorem \ref{ThmMain}} \label{SecProofJointConv}


Prior to dealing with joint convergence of the infinite dimensional vectors in the statement of the theorem we justify that under the assumptions of part (b) and as $n\to \infty$
\begin{equation*}
	\sqrt{n} \, \big( \hat q_{n} - q \big) \to^d N\big(0, \sigma^2_{q} \big) \quad \mbox{ and } \quad \sqrt{n} \, \big( \hat p_{n} - p \big) \to^d N\big(0, \sigma^2_{p} \big),
\end{equation*}
where, recalling that $f^{(q)}:=\sum_{j\in\Z\setminus\{0\}} f^{(q_j)}$ and $f^{(p)}:=\sum_{j\in\Z\setminus\{0\}} f^{(p_j)}=\lambda^{-1}(f^{(q)} - pf^{(\lambda)})$,
\begin{align*}
	\sigma^2_{q}\! :=\! \frac{1}{{\Deltaup}^2 }\!\! \int_{\R}\!\! \Big(\! f^{(q)} \! \ast \! {\cal F}^{-1}\! \left[ \varphi^{-1}(-\cdot) \right] (x)\!\Big)^2\!\! P(dx) \, \, \mbox{ and } \, \,
	\sigma^2_{p}\! :=\! \frac{1}{{\Deltaup}^2}\!\! \int_{\R}\!\! \Big(\! f^{(p)}\! \ast\! {\cal F}^{-1}\! \left[ \varphi^{-1}(-\cdot) \right] (x)\!\Big)^2 \!\! P(dx). 
\end{align*}
The second set of assumptions of Theorem \ref{ThmCLTEstimators} are satisfied here and, in view of the expression for $\hat q_n$ in Section \ref{SecSettingEstimators}, the first result thus follows by taking $\mathcal{T}=\emptyset$, $C_{(\lambda)}=C_{(\gamma)}=0$ and $C_{(j)}=1$ for all $j\in \Z\setminus\{0\}$ in Theorem \ref{ThmCLTEstimators}. For the second result we write 
\begin{equation*}
	\sqrt{n} \, \big( \hat p_{n} - p \big) = \hat{\lambda}_n^{-1} \, \sqrt{n} \, \big( (q_n - q) + p \, (\lambda - \hat \lambda_n) \big).
\end{equation*}
Taking $\mathcal{T}=\emptyset$, $C_{(\lambda)}=-p$, $C_{(\gamma)}=0$ and $C_{(j)}=1$ for all $j\in \Z\setminus\{0\}$ in Theorem \ref{ThmCLTEstimators}, the quantity by which $\hat \lambda_n^{-1}$ is multiplied on the right hand side converges to $N\big(0, \lambda^2\sigma_{p}^2\big)$. Since $\hat \lambda_n$ converges to $\lambda$ (constant) by the same theorem, the conclusion follows by Slutsky's lemma.

We now prove parts (a) and (b) together under the respective assumptions. From Lemma 1.4.8 in \cite{AvdVW96} we have that joint convergence of the infinite dimensional vectors in both parts follows if we show joint convergence of all their finite dimensional projections. To show the latter let $\delta^{i,j}$ denote the Kronecker delta, i.e. the mapping from $\Z \times \Z$ to $\{0,1\}$ that equals $1$ only if $i=j$. Define $\boldsymbol{\delta}:=\big( \delta^{1,m_\lambda}, \delta^{1,m_\gamma}, \delta^{1,m_q}, \delta^{1,m_p}, \boldsymbol \delta^{(q)}, \boldsymbol \delta^{(p)}, \delta^{1,m_{N}}, \delta^{1,m_F} \big)$, where, for any $j\in \Z \setminus\{0\}$, $\boldsymbol \delta^{(q)}_j=\delta^{1,m_{q_j}}, \boldsymbol \delta^{(p)}_j=\delta^{1,m_{p_j}}$ and $m_\lambda, m_\gamma, m_q, m_p, m_{q_j}, m_{p_j}, m_{N}, m_F\in \{0,1\}$ are such that
\begin{equation*}
	m_\lambda+ m_\gamma+ m_q+ m_p+ M_q + M_p + m_{N} + m_F  \in \N, \qquad \mbox{with} \quad M_{q}:=  \sum_{\mathclap{j\in \Z\setminus\{0\}}} m_{q_j} \quad \mbox{and} \quad M_{p}:=\sum_{\mathclap{j\in \Z\setminus\{0\}}} m_{p_j}.
\end{equation*}
Then, writing $\cdotp$ for the coordinate-wise product of two infinite vectors, we denote joint convergence of a finite dimensional projection as having that as $n\to \infty$
\begin{equation*}
	\sqrt{n} \, \,  \boldsymbol \delta \, \, \cdotp \left(
	\hat \lambda_n -\lambda,\,
	h_n^{-1}(\hat \gamma_n-\gamma),\,
	\hat q_n - q,\,
	\hat p_n - p,\,
	\hat{\boldsymbol{q}}_n-\boldsymbol q ,\,
	\hat{\boldsymbol{p}}_n-\boldsymbol p ,\,
	\widehat{N}_n - N ,\,
	\widehat{F}_n - F 
	\right)
	\to^{\mathcal{L}^{\times, \delta}} \boldsymbol \delta \,  \cdotp \mathbb{L},
\end{equation*}
where $\to^{\mathcal{L}^{\times, \delta}}$ means convergence in distribution in the corresponding product space
\begin{equation*}
	\mathbb{D}=\mathbb{D}\big(m_\lambda, m_\gamma, m_q, m_p, (m_{q_j})_{j\in \Z\setminus\{0\}},  (m_{p_j})_{j\in \Z\setminus\{0\}}, m_{N}, m_F\big).
\end{equation*}
Throughout we fix all these binary parameters. To show the joint convergence displayed above we note that, under the assumptions of each of the two parts of the theorem, marginal convergence of each of the non-zero coordinates holds by Propositions \ref{Proplambdamu} and \ref{Proppsqs}, Theorem \ref{ThmNF} and the calculations regarding $q$ and $p$ at the beginning of the proof. Therefore, the sequence given by each non-zero projection is asymptotically tight and asymptotically measurable in the respective space. Then, by Lemmas 1.4.3 and 1.4.4 in \cite{AvdVW96}, the sequence of random variables given by the finite dimensional vector above taking values in $\mathbb{D}$ is asymptotically tight and asymptotically measurable. By Prohorov's theorem it is relatively compact, i.e. every subsequence has a further weakly convergent subsequence, so to finish the proof it suffices to show that all limits are the same. Denote by $\mathcal{H}$ the linear span of the functions $H: \R^{\N} \times \ell^{\infty}(\R)^2 \to \R$ of the form
\begin{align*}
	H(\boldsymbol L)= & \, \delta^{1,m_\lambda}  h^{(\lambda)}(\boldsymbol L_\lambda) + \delta^{1,m_\gamma} h^{(\gamma)}(\boldsymbol L_\gamma) + \delta^{1,m_q} h^{(q)}(\boldsymbol L_q) + \delta^{1,m_p}  h^{(p)}(\boldsymbol L_p) + \!\!\! \sum_{j\in \Z \setminus\{0\}} \!\! \!\delta^{1,m_{q_j}} h^{(q_j)}(\boldsymbol L_{q_j})\\
	& + \sum_{j\in \Z \setminus\{0\}}  \delta^{1,m_{p_j}} h^{(p_j)}(\boldsymbol L_{p_j}) + \delta^{1,m_{N}} \sum_{i=1}^{M_{N}} h^{(i)}  \big( \boldsymbol L_{N} (t_{i}) \big) + \delta^{1,m_{F}} \sum_{i=M_{N}+1}^{M_{N}+M_F} h^{(i)}  \big( L_{F} (t_{i}) \big) 
\end{align*}
for any $M_{N}, M_F\in \N$, $t_i \in \R$ and $h^{(\cdot)}\in C_b(\R)$ fixed throughout, where if a sum is empty it equals $0$ by convention. Then, for any $m_\lambda, m_\gamma, m_q, m_p, (m_{q_j})_{j\in \Z\setminus\{0\}},  (m_{p_j})_{j\in \Z\setminus\{0\}}, m_{N}, m_F$ fixed, $\mathcal{H}\subset C_b\big(\mathbb{D} \big)$ is a vector lattice containing the constant functions and separating points of $\mathbb{D}$ (see footnote $\flat$ of page 25 in \cite{AvdVW96} for the definition of these terms). We claim that, for any of the parameters here introduced and under the corresponding assumptions depending on whether $m_q+m_p=0$ or not, as $n\to \infty$
\begin{equation*}
	\sqrt{n} \left(
	\begin{array}{c}
		\boldsymbol \delta_{-N, -F} \, \cdotp \left(
		\begin{array}{c}
			\hat \lambda_n -\lambda\\
			h_n^{-1}(\hat \gamma_n-\gamma)\\
			\hat{q}_n- q \\
			\hat{p}_n-p \\
			\hat{\boldsymbol{q}}_n-\boldsymbol q \\
			\hat{\boldsymbol{p}}_n-\boldsymbol p \\
		\end{array}
		\right) \\
		\delta^{1,m_{N}} \big(\boldsymbol{\widehat{N}_n} - \boldsymbol{N} \big)\\
		\delta^{1,m_{F}} \big(\boldsymbol{\widehat{F}_n} - \boldsymbol{F} \big)
	\end{array}
	\right)
	\to \left(
	\begin{array}{c}
		\boldsymbol \delta_{-N, -F} \, \cdotp \mathbb{L}_{-N, -F} \\
		\delta^{1,m_{N}}  \left(
		\begin{array}{c}
			\mathbb{L}_{N}(t_1) \\
			\vdots \\
			\mathbb{L}_{N}(t_{M_{N}}) 
		\end{array}
		\right) \\
		\delta^{1,m_{F}}  \left(
		\begin{array}{c}
			\mathbb{L}_{F}(t_{M_{N}+1}) \\
			\vdots \\
			\mathbb{L}_{F}(t_{M_{N}+M_{F}}) 
		\end{array}
		\right)  
	\end{array}
	\right)
\end{equation*}
in distribution in $\R^{m}$, $m:=m_\lambda+ m_\gamma+ m_q+ m_p+ M_p + M_q + \delta^{1,m_{N}} M_{N}+
\delta^{1,m_{F}} M_{F}$, where subscript ${\scriptstyle \bullet}_{-N, -F}$ denotes vector ${\scriptstyle \bullet}$ without its last two coordinates, $\boldsymbol{N}=\big( N(t_1), \ldots, N(t_{M_{N}}) \big), \boldsymbol{F}=\big( F(t_{M_{N}+1}), \ldots, F(t_{M_{N}+M_{F}}) \big)$ and $\boldsymbol{\widehat{N}_n}, \boldsymbol{\widehat{F}_n}$ are the respective coordinate-wise estimators. Then the continuous mapping theorem together with Lemma 1.3.12 (ii) in \cite{AvdVW96} justify the joint convergence we are seeking. To show the claim we note that by the Cram\'er--Wold device it is sufficient to check that any linear combination of the coordinates on the left hand side converges to the same combination of the coordinates on the right hand side. For any $c_{(\lambda)}, c_{(\gamma)}, c_{(q)}, c_{(p)} \in \R$ any row vectors $\boldsymbol c_{(q)}, \boldsymbol c_{(q)} : \Z \setminus \{0\} \to \R^{\N}$ with respective $j$-th entries $c_{(q_j)}$ and $c_{(p_j)}$, any $\boldsymbol C_{N} =(c_1, \ldots, c_{M_{N}} )$ and $\boldsymbol C_{F} =(c_{M_{N}+1}, \ldots, c_{M_{N}+M_F} )$, the obvious linear combination of the left hand side arising from these parameters can be written as
\begin{equation*}
	\sqrt{n} \big( c_{(\lambda)}, c_{(\gamma)}, c_{(q)},  \boldsymbol c_{(q)},  \boldsymbol C_{N} \big) \!\!\left( \!\!\!
	\begin{array}{c}
		\delta^{1,m_\lambda} (\hat \lambda_n \!-\lambda)\\
		\delta^{1,m_\gamma} (h_n^{-1}(\hat \gamma_n\!-\!\gamma))\\
		\delta^{1,m_q} (\hat{q}_n\!- q )\\
		\boldsymbol \delta^{(q)}  \cdotp (\hat{\boldsymbol{q}}_n\!\!-\boldsymbol q) \\
		\delta^{1,m_{N}} (\boldsymbol{\widehat{N}}_n \!\!- \boldsymbol{N} )
	\end{array}
	\!\!\!\right)
	\!+ \hat{\lambda}_n^{-1} \! \sqrt{n} \big( \!-\tilde{c}_\lambda , c_{(p)},  \boldsymbol c_{(p)}, \boldsymbol C_{F}  \big)\!\!\left(\!\!
	\begin{array}{c}
		(\hat \lambda_n \!-\lambda)\\
		\delta^{1,m_p} (\hat{q}_n\!- q )\\
		\boldsymbol \delta^{(p)}  \cdotp (\hat{\boldsymbol{q}}_n\!\!-\boldsymbol q) \\ 
		\delta^{1,m_{F}} (\boldsymbol{\boldsymbol{\widehat{N}}}_n' \!\!- \boldsymbol{\boldsymbol{N}}' )\\
	\end{array}
	\!\!\right)\!\!,
\end{equation*}
where
\begin{equation*}
	\tilde{c}_{\lambda}:=\delta^{1,m_p} c_{(p)} p + \sum_{j\in \Z\setminus\{0\}} \delta^{1,m_{p_j}} c_{(p_j)} p_j + \delta^{1,m_{F}} \sum_{i=M_{N}+1}^{M_N+M_F} c_i F(t_i), 
\end{equation*}
$\boldsymbol{N}'=\big( N(t_{M_{N} +1}), \ldots, N(t_{M_{N} +M_{F}}) \big)$ and $\boldsymbol{\widehat{N}}_n'$ is its coordinate-wise estimator. To justify that the last display converges to the correct linear combination of limiting distributions we first note that by Theorem \ref{ThmCLTEstimators} the finitely many non-zero coordinates in the column vectors of the last display converge jointly to the vector comprising their respective limits. The conclusion then follows by the continuous mapping theorem and in view of the explicit representation of the variance of the limiting random variable given by Theorem \ref{ThmCLTEstimators} when  $\mathcal{T}=\{ t_1, \ldots, t_{M_{N}+M_F} \}$, $C_{(\lambda)}=\delta^{1,m_{\lambda}}c_{(\lambda)} - \lambda^{-1} \tilde{c}_{\lambda}$, $C_{(\gamma)}=c_{(\gamma)}$, $C_{(j)}= \delta^{1,m_{q}} c_{(q)} + \delta^{1,m_{q_j}} c_{(q_j)} + \lambda^{-1} \big(\delta^{1,m_{p}} c_{(p)} + \delta^{1,m_{p_j}} c_{(p_j)} \big)$, $C_{t}=\delta^{1,m_{N}} c_{t}$ if $t=t_1, \ldots, t_{M_{N}}$ and $C_{t}=\lambda^{-1} \delta^{1,m_{F}} c_{t}$ if $t=t_{M_{N}+1}, \ldots, t_{M_{N}+M_F}$.

\subsection{Proof of Lemma \ref{lemmaAsympSigma}}\label{SecProofLemma}

This lemma follows immediately from the expressions for $\FII{\varphi^{-1}(-\cdot)} $ and $P$ given in \eqref{eqFIPhim1nP}. In view of these and of the observation after \eqref{eqFinVar}, we have that for any $f$ and $g$ bounded in $\R$
\begin{align*}\label{Rep1}
	\int_{\R}\! & \Big(f \ast  {\cal F}^{-1} \left[ \varphi^{-1}(-\cdot) \right] (x)  \Big) \Big( g  \ast {\cal F}^{-1} \left[ \varphi^{-1}(-\cdot) \right] (x)\Big)\, P(dx) \notag \\
	& = e^{\lambda {\Deltaup}} \bigg(f (0) + f \ast \sum_{k=1}^\infty \bar \nu^{\ast k} \frac{(-{\Deltaup})^{k}}{k!} (0)  \bigg) \bigg( g(0) + g  \ast \sum_{k=1}^\infty \bar \nu^{\ast k} \frac{(-{\Deltaup})^{k}}{k!} (0)\bigg)  \\
	& \quad \, + e^{\lambda {\Deltaup}} \! \bigintssss_{\R}  \! \bigg( f(x) + f \ast \sum_{k=1}^\infty \bar \nu^{\ast k} \frac{(-{\Deltaup})^{k}}{k!} (x)  \bigg)  \bigg( g(x) + g  \ast \sum_{k=1}^\infty \bar \nu^{\ast k} \frac{(-{\Deltaup})^{k}}{k!} (x)\bigg)   \sum_{k=1}^{\infty} \nu^{\ast k} \frac{{\Deltaup}^k}{k!} (dx) \\
	& =   f(0) \,  g(0) + {\Deltaup} \! \left( \int_{\R} f(x) g(x) \, \nu(dx) - f(0) \, g \ast \bar \nu \, (0) - g(0) \, f \ast \bar \nu \, (0) + \lambda \, f(0)  g(0) \! \right) \! +\! O\big((\lambda {\Deltaup})^2\big).
\end{align*}




\end{document}